\documentclass[10pt,oneside,reqno]{amsart}
\usepackage{amsfonts,amsmath,amsthm}
\usepackage{amssymb,epsfig}
\usepackage{pgf,tikz}
\usepackage{fancybox}
\usepackage[T1]{fontenc}
\usepackage{stmaryrd}
\usepackage{graphicx}
\usepackage{eso-pic}
\usepackage[top=2cm, bottom=2cm, outer=0cm, inner=0cm, footskip=0.35in]{geometry}
\usepackage{fullpage}

\usepackage[unicode=true]{hyperref}
\hypersetup{colorlinks = true}
\hypersetup{
    colorlinks = true,
    linkcolor = {blue},
    citecolor={blue}
}

\usepackage[above,below,verbose,section]{placeins}
\makeatletter
\usepackage{mathrsfs}
\usepackage{mathpazo}
\usepackage{lastpage}
\usepackage{enumitem}
\usepackage{tikz}
\usetikzlibrary{shapes.geometric, arrows, calc}
\tikzstyle{startstop} = [rectangle, rounded corners, minimum width=4cm, minimum height=2cm,text centered, draw=black, fill=red!30]
\tikzstyle{arrow} = [thick,->,>=stealth]
\makeatletter
\def\namedlabel#1#2{\begingroup
    #2%
    \def\@currentlabel{#2}%
    \phantomsection\label{#1}\endgroup
}
\makeatother

\usepackage[colorinlistoftodos,prependcaption]{todonotes}
\presetkeys%
{todonotes}%
{inline,backgroundcolor=white}{}
\tikzset{/tikz/notestyleraw/.append style={text=black}}

\newtheorem{thm}{Theorem}[section]

\newtheorem{lem}[thm]{Lemma}
\newtheorem{defn}[thm]{Definition}
\newtheorem{prop}[thm]{Proposition}
\newtheorem{cor}[thm]{Corollary}
\newtheorem{que}{Open Problem}

\newtheorem{rmk}[thm]{Remark}
\newcommand{\be}{\begin{eqnarray}}
\newcommand{\ee}{\end{eqnarray}}
\newcommand{\ben}{\begin{eqnarray*}}
\newcommand{\een}{\end{eqnarray*}}
\newcommand{\beq}{\begin{equation}}
\newcommand{\eeq}{\end{equation}}
\newcommand{\beal}{\begin{aligned}}
\newcommand{\enal}{\end{aligned}}

\newcommand{\eps}{\varepsilon}

\newcommand{\lb}{\lambda}

\newcommand{\T}{\mathbb{T}}
\newcommand{\R}{\mathbb{R}}

\newcommand{\N}{\mathbb{N}}

\newcommand{\om}{\omega}
\newcommand{\Om}{\Omega}

\newcommand{\dt}{\delta}

\newcommand{\cE}{\mathcal{E}}

\newcommand{\cM}{\mathcal{M}}

\newcommand{\cC}{\mathcal{C}}

\newcommand{\cP}{\mathcal{P}}

\newcommand{\cF}{\mathcal{F}}

\newcommand{\cU}{\mathcal{U}}

\newcommand{\wt}{\widetilde }
\newcommand{\ol}{\overline }

\title[nonlinear H-J equations with state-constraint]{\scshape \textsc{Generalized convergence of solutions for nonlinear Hamilton--Jacobi equations with state-constraint}}

\author{Son N.T. Tu$^\dagger$ and  Jianlu Zhang*}

\address{Son. N. T. Tu\\ $^\dagger$Department of Mathematics, Michigan State University\\
East Lansing, Michigan 48824, USA}
    \email{tuson@msu.edu}

\address{Jianlu Zhang\\ *Hua Loo-Keng Key Laboratory of Mathematics \& Mathematics Institute
            \\Academy of Mathematics and systems science
            \\Chinese Academy of Sciences, Beijing 100190, China}
        \email{jellychung1987@gmail.com}

\subjclass[2010]{35B40,\ 37J50,\ 49L25}
\keywords{Hamilton-Jacobi equations, state-constraint problems, Mather measures, ergodic constant, viscosity solution}
\thanks{{\it Statements and Declarations:} The authors declare no conflict of interest.}
\thanks{{\it Data Availability:} We have no data associate for the paper.}

\date{\today}

\begin{document}
\maketitle



\begin{abstract} 
For a continuous Hamiltonian $H : (x, p, u) \in T^*\mathbb{R}^n \times \mathbb{R}\rightarrow \mathbb{R}$, we consider the asymptotic behavior of associated Hamilton--Jacobi equations with state-constraint
\begin{equation*}
    \begin{cases}
        \begin{aligned}
        H(x, Du, \lambda u)\leq C_\lambda, \quad x\in\Omega_\lambda\subset \mathbb{R}^n,\\
        H(x, Du, \lambda u)\geq C_\lambda, \quad x\in\overline{\Omega}_\lambda\subset \mathbb{R}^n,
        \end{aligned}
    \end{cases}
\end{equation*}
as $\lambda\rightarrow 0^+$. When $H$ satisfies certain convex, coercive and monotone conditions, the domain $\Omega_\lambda:=(1+r(\lambda))\Omega$ keeps bounded, star-shaped 
for all $\lambda>0$ with $\lim_{\lambda\rightarrow 0^+}r(\lambda)=0$, and $\lim_{\lambda\rightarrow 0^+}C_\lambda=c(H)$ equals the ergodic constant of $H(\cdot,\cdot,0)$, we 
prove the convergence of solutions $u_\lambda$ to a specific solution of the critical equation
\begin{equation*}
    \begin{cases}
        \begin{aligned}
            H(x, Du, 0)\leq c(H), \quad x\in\Omega\subset\mathbb{R}^n,\\
            H(x, Du, 0)\geq c(H), \quad x\in\overline{\Omega}\subset\mathbb{R}^n.
        \end{aligned}
    \end{cases}
\end{equation*}
 We also discuss the generalization of such a convergence for equations with more 
 general $C_\lambda$ and $\Omega_\lambda$.
\end{abstract}

\section{Introduction}

Let $\Om\subset\R^n$ be a bounded open set and $T^*\Om$ be its cotangent bundle. For any continuous function  $H : T^* \Om\times \R\rightarrow \R$, the so called  {\it contact Hamiltonian}, we are interested in the {\it nonlinear Hamilton--Jacobi (H-J) equation} with state-constraint:
\beq
\left\{
\begin{aligned}
    H(x, Du,  u)\leq c, \quad x\in\Om\subset\R^n,\\
    H(x, Du, u)\geq c, \quad x\in\ol{\Om}\subset\R^n.
\end{aligned}
\right.
\eeq
The {\it viscosity solutions} for such a equation under mild assumptions on $H$, $\Om$ and $c\in\R$ has been studied deeply in the literature, e.g., \cite{C-DL, L, Mitake2008, MT, S1, S2} and the references therein. Typically, a {\it comparison principle} is used to ensure the uniqueness of the viscosity solution. However, this principle does not hold in critical cases:
\begin{equation}\label{eq:hj-c}
    \left\{
    \begin{aligned}
        H(x, Du, 0)\leq c, \quad x\in\Om\subset\R^n,\\
        H(x, Du, 0)\geq c, \quad x\in\ol{\Om}\subset\R^n.
    \end{aligned}
    \right.
\end{equation}
Besides, the critical H-J equation \eqref{eq:hj-c} is solvable only for a unique value of $c\in\R$ (so called {\it ergodic constant} {\rm or} {\it Ma\~n\'e's critical value}),  which is defined by
\beq\label{eq:mane-cri}
    c(H):=\inf\{c\in\R\;|\;H(x,D u(x),0)\leq c \text{ possesses a solution in $\Om$}\}.
\eeq
Since \eqref{eq:hj-c} may possess multiple solutions (even up to additive constants), a selection principle should be explored, to pick up certain distinguished solution of the critical equation, as a limit of a family of solutions of the nonlinear H-J equations with respect to certain parameters. Precisely, we consider the following parametrized equation 
\beq\label{eq:hj-lb}
\quad\left\{
\begin{aligned}
    H(x,D  u(x),\lb u(x))\leq C_\lb,\quad x\in \Om_\lb,\\
    H(x,D  u(x),\lb u(x))\geq C_\lb,\quad x\in \ol\Om_\lb,\\
\end{aligned}
\right.
\eeq
with both $C_\lb\in\R$ and $\Om_\lb\subsetneq\R^n$ uniformly bounded with respect to $\lb>0$. Once 
\begin{equation*}
    \lim_{\lb\rightarrow 0^+}\Om_\lb=\Om, \qquad \lim_{\lb\rightarrow 0^+}C_\lb=c(H)
\end{equation*}
in some sense, we aim to prove the convergence of the solution $u_\lb$ of \eqref{eq:hj-lb} to a specific solution $u_0$ of 
\begin{equation}\label{eq:hj-0}
    \quad
    \left\{
    \begin{aligned}
        H(x,D  u(x),0)\leq c(H),\quad x\in \Om,\\
        H(x,D  u(x),0)\geq c(H),\quad x\in \ol\Om,
    \end{aligned}
    \right.
\end{equation}
as $\lb\rightarrow 0^+$. The choice $C_\lambda \to c(H)$ is chosen such that the family of the corresponding solutions is uniformly bounded without necessary normalization (see Theorem \ref{thm:change0-domain-c-ener} for the general case). 

Such a problem was initiated by Lions, Papanicolaou and Varadhan in 1987. In their unpublished  seminal paper \cite{LPV}, they firstly derived the ergodic costant $c(H)$ by a homogenization approach for critical H-J equations on the torus $\T^n$. They also proposed a {\it vanishing discount} scheme to get $c(H)$, namely, for the solution $u_\lb$ of the {\it discounted equation} 
\begin{equation*}
    \lb u+H(x,Du)=0,\quad\ x\in\T^n, \,\lb>0,
\end{equation*}
there holds $\lim_{\lb\rightarrow 0^+}\lb u_\lb=-c(H)$. As a sequel, the convergence of $u_\lb+\lambda^{-1}c(H)$ as $\lb\rightarrow 0^+$ was finally solved in \cite{DFIZ}, by relating the  convergence of $u_\lb$ to the asymptotic behavior of certain {\it Mather measures} from {\it weak KAM theory}. As a matter of fact, partial clues have been noticed in earlier works \cite{Gom08,ISM}. It is remarkable that in \cite{CCIZ, IMT1,IMT2}, a duality representation of the Mather measure was independently obtained, which was later used in wide vanishing-discount problems, e.g. \cite{IS} for $\R^n$ space, \cite{Tu} for discounted state-constraint equation, \cite{ishii_vanishing_2020-system} for discounted system of equations and \cite{IMT1,IMT2, MT} for second order Hamilton--Jacobi equations.\smallskip

In our case, the Hamiltonians depends on $u$ nonlinearly, so it is not easy to use the duality representation and we have to resort to a different approach to obtain the Mather measures. The first attempt was made in \cite{WYZ} for $C^2-$smooth contact H-J equations on closed manifolds, which was later generalized in \cite{Ch,CFZZ,WYZ2, Z} under wider conditions.\smallskip


In the current paper we aim to explore the asymptotic behavior of $u_\lb$ as $\lb\rightarrow 0^+$, for nonlinear H-J equations with state-constraints. The main difference of our setting comparing to previous works \cite{Ch, Tu, WYZ} is the totally variable $\Om_\lb$ and $C_\lb$. Moreover, as we will see, the Hamiltonian's nonlinear dependence on $u$ leads to a more complicated expression of  $\lim_{\lb\rightarrow 0^+}u_\lb$, which is deeply relevant to the deformation of Mather measures with respect to the changing rate of $\Om_\lb$ as $\lambda\to 0^+$. That in turn supplies a criterion for the limit of $(C_\lambda - c(H))/\lambda$ as $\lambda\rightarrow 0^+$ (see subsection \ref{ss}). \smallskip

Recall that for discounted Hamiltonians,  the convergence of solution $u_\lambda$ of \eqref{eq:hj-lb} was proved for two important cases: $C_\lambda = c(H)$ and $C_\lambda = c(\lambda)$ (the ergodic constant of $H(\cdot,\cdot,0)$ on $\Omega_\lambda$) in \cite{Tu}. Besides, the regularity of $c(\lb)$ with respect to $\lambda$ is also studied in \cite{Tu}. This paper extends the results in \cite{Tu} from two aspects: $H$ is nonlinearly dependent of $u$; and more general $C_\lb$, by using the nonlinear tools developed from \cite{WWY, WYZ}.

\smallskip

Finally, let us remark that in the context of PDEs, extensive work has been done on the well-posedness of H-J equations and fully nonlinear elliptic equations with state constraints, particularly for the discounted setting. The problem for the convex first-order setting was initially studied in \cite{S1, S2} using optimal control techniques. The general first-order problem, including the nonconvex case, was addressed in \cite{C-DL}. For more classical properties of first-order state-constraint problems, we refer to \cite{Bardi1997, Barles1994, tran_hamilton-jacobi_2021}, and the references therein. For second-order state-constraint problems, we refer to \cite{YuTu2021large, lasry_nonlinear_1989}. Additionally, a similar investigation into the regularity of the ergodic constant for some second-order PDEs has been conducted in \cite{barles_large_2010, BKT2022}. \smallskip

We conclude with a brief overview of some related boundary value problems, though this list is not exhaustive. For discounted problems, periodic, state-constraint, Dirichlet, and Neumann type boundary conditions are considered in \cite{IMT1, IMT2}. Fully nonlinear settings for Dirichlet and Neumann type boundary conditions are addressed in \cite{barles_fully_1993, barles_fully_1991, ishii_viscosity_1990, lions_neumann_1985, patrizi_neumann_2008}. Related oblique derivative boundary value problems are studied in \cite{dupuis_oblique_1990, ishii_fully_1991, lions_linear_1986, CRMATH_2002__334_8_661_0}. Asymptotic behaviors of solutions to various boundary conditions (besides the state-constraint mentioned above) are considered in \cite{barles_large_2012, barles_pde_2012, giga_singular_2014, ishii_long-time_2011}. Some aspects of weak KAM theory for Neumann type boundary conditions are considered in \cite{ishii_weak_2011, serea_viscosity_2007}.

\subsection{Setting}\label{ss1}
In this paper, $\Om$ is meant by an open, bounded, connected subset of $\R^n$. Without loss of generality, we always assume $0\in\Om$ and the following:

\begin{description}[style=multiline, labelwidth=1.1cm, leftmargin=2cm]
    \item[\namedlabel{itm:C0''}{$(\mathcal{C}_0)$}] $\Omega$ is star-shaped (with respect to the origin) and there exists $\theta>0$ such that $\mathrm{dist}(x,\overline{\Omega})>\theta \varepsilon$ for all $x\in (1+\varepsilon)\partial \Omega$ for all $\varepsilon>0$.
\end{description}
We remark that the condition $\Omega$ is star-shaped can be removed in \ref{itm:C0''} (see \cite{Tu}). For the sake of variable domains, we always take 
\[
\Om_\lb:=(1+r(\lb))\Om
\]
with continuous $r:\lb\in(0,+\infty)\rightarrow\R$ tending to zero as $\lb\rightarrow 0^+$ such that

\begin{description}[style=multiline, labelwidth=1.1cm, leftmargin=2cm]
    \item[\namedlabel{itm:L}{$(\mathcal{L})$}] $\lim_{\lambda\rightarrow 0^+} r(\lambda)/\lambda=\eta\in\mathbb{R}$.
\end{description}

Without loss of generality, we can assume the existence of an open, connected  and bounded set $\cU\subset\R^n$, such that  $\ol{\Om}_\lb\subset\cU$ for all $\lb>0$. Consequently, the Hamiltonian $H:T^*\cU\times\R$ is assumed to be a continuous function satisfying the following conditions:
\begin{description}[style=multiline, labelwidth=1.1cm, leftmargin=2cm]
     \item[\namedlabel{itm:C1'}{$(\mathcal{C}_1)$}] There exists $\kappa>0$ such that $H(x,p,u)-\kappa u$ is nondecreasing in $u$ for any $(x,p)\in T^*\ol\cU$.
    \item[\namedlabel{itm:C2}{$(\mathcal{C}_2)$}] For any $(x,u)\in \cU\times\R$, the function $H(x,\cdot,u)$ is convex in $p\in T_x^*\cU$.
    \item[\namedlabel{itm:C3}{$(\mathcal{C}_3)$}] $\lim_{|p|\rightarrow+\infty}H(x,p,u)=+\infty$ for any $(x,u)\in\cU\times\R$.
    \item[\namedlabel{itm:C4}{$(\mathcal{C}_4)$}] There exists a {\it local modulus}\footnote{A {\it modulus} $\varpi:x\in [0,+\infty)\rightarrow\R$ is a continuous nondecreasing, nonnegative, subadditive function such that $\varpi(0)=0$. A {\it local modulus} $\varpi :(x,y)\in [0,+\infty)\times[0,+\infty)\rightarrow\R$ is a modulus in $y$ for each $x\geq 0$,  continuous in $(x,y)$ and nondecreasing with respect to $x$.} $\varpi_1:[0,+\infty)\times[0,+\infty)\rightarrow[0,+\infty)$ such that 
    \begin{equation*}
        |H(x,p,u)-H(y,p,u)|\leq \varpi_1\big(|u|, |x-y|(|p|+1)\big)
    \end{equation*}
    for all $x,y\in\ol\cU,\; p\in T^*_x\ol\cU,\; u\in\R$.
    \item[\namedlabel{itm:C5}{$(\mathcal{C}_5)$}] There exists a local modulus $\varpi_2:[0,+\infty)\times[0,+\infty)\rightarrow[0,+\infty)$ such that 
    \begin{equation*}
        |H(x,p,u)-H(x,q,u)|\leq \varpi_2(|u|, |p-q|)
    \end{equation*}
    for all $(x,u)\in\ol\cU\times\R, p,q\in T_x^*\ol\cU$.
    \item[\namedlabel{itm:C6}{$(\mathcal{C}_6)$}] $\partial_u H(x,p,0)$ exists and there exists a local modulus $\varpi_3:[0,+\infty)\times[0,+\infty)\rightarrow[0,+\infty)$ such that 
    \begin{equation*}
        \Big|\frac{H(x,p,u)-H(x,p,0)}{u}-\partial_u H(x,p,0)\Big|\leq \varpi_3(|p|, |u|)
    \end{equation*}
    for any $(x,p)\in T^*\overline{\cU}$.
    \item[\namedlabel{itm:C7}{$(\mathcal{C}_7)$}] $\partial_x H(x,p,0)$ exists and there exists a local modulus $\varpi_4:[0,+\infty)\times[0,+\infty)\rightarrow[0,+\infty)$ such that 
    \begin{equation*}
        \Big|\frac{H(x,p,0)-H((1+\dt)x,p,0)}{\dt}+\langle\partial_x H(x,p,0),x\rangle\Big|\leq \varpi_4(|p|, |\dt|)
    \end{equation*}
    for any $(x,p)\in T^*\overline{\cU}$.
\end{description}

Throughout the paper, unless otherwise stated,  solution, subsolution, supersolution  are meant by the viscosity solution, viscosity subsolution, viscosity supersolution, respectively.

\subsection{Main results} For a Hamiltonian $H$ satisfying \ref{itm:C2} and \ref{itm:C3}, the associated {\it Lagrangian} can be defined by 
\begin{equation*}
    L(x,v,u) := \sup_{p\in T^*_x\mathcal{U}} \Big(\langle p,v\rangle - H(x,p,u)\Big), \qquad 
     (x,v,u)\in T\overline{U}\times\R.
\end{equation*}
In our setting, we can always strengthen assumption \ref{itm:C3} to a stronger condition on super-linearity (see Remark \ref{rmk:superlinear} for the reason). If so, the Lagrangian $L$ is finite. Besides, $L$ will possess the corresponding assumptions as $H$ has in subsection \ref{ss1}. \smallskip

As the generalization of conclusions in \cite{Ch, WYZ}, we first state the following conclusion for a fixed domain:
\begin{thm}
\label{thm:fix}
Let $\Omega$ satisfy \ref{itm:C0''} and $H$ satisfy \ref{itm:C1'}--\ref{itm:C6}. Let $\vartheta_\lambda \in \mathrm{C}(\overline{\Omega})$ be the unique solution to
\begin{equation}\label{eq:thm:fix-ge-equation}
\begin{cases}
    \begin{aligned}
    H\left(x,D  \vartheta_\lambda(x),\lambda\vartheta_\lambda(x)\right)\leq c(H),\quad x\in \Omega,\\
    H\left(x,D  \vartheta_\lambda(x),\lambda\vartheta_\lambda(x)\right)\geq c(H),\quad x\in \overline{\Omega}.
    \end{aligned}
\end{cases}
\end{equation}
Then $\vartheta_\lambda$ converges to a viscosity solution $u_0\in \mathrm{C}(\overline{\Omega})$ of \eqref{eq:hj-0}. Moreover, 
\begin{equation*}
    u_0=\sup\mathcal{E}
\end{equation*}
where 
\begin{equation*}
\begin{aligned}
    \cE& := &\left \{w\in \mathrm{C}(\Omega)\;\text{is a subsolution of \eqref{eq:hj-0}}\;\Big|\int_{T\ol\Om}\partial_u L(x,v,0)w(x)d\mu(x,v)\geq 0,\;\forall \mu\in\cM\right\}
\end{aligned}
\end{equation*}
and $\mathcal{M}$ is the set of Mather measures (defined in Definition \ref{defn:Mather-measures}).
\end{thm}

\begin{rmk}
Assumptions \ref{itm:C0''} and \ref{itm:C1'} are used to ensure the validity of the  comparison principle, which further guarantees the uniqueness of solution for \eqref{eq:hj-lb}. We emphasize that the state-constraint boundary creates additional technical difficulties compared to the case in \cite{Ch,WYZ}, so \ref{itm:C0''} is necessary. In this paper, we provide an independent proof of the comparison principle (see Appendix). We refer also to \cite{C-DL,S1} and \cite[Theorem 2.8.7]{tu_asymptotic_thesis} for different versions of the comparison principle under similar conditions as \ref{itm:C0''}.  
\end{rmk}

Next, we state our main result regarding equation \eqref{eq:hj-lb}: 



\begin{thm}
\label{thm:ge}
Suppose $\Omega$ satisfies \ref{itm:C0''}, $r(\lambda)$ satisfies \ref{itm:L}, the Hamiltonian $H$ satisfies \ref{itm:C1'}--\ref{itm:C7} and  $u_\lambda\in \mathrm{C}(\overline{\Omega}_\lambda)$ is the solution to \eqref{eq:hj-lb} with $C_\lambda \to c(H)$ as $\lambda \to 0^+$. If 
\begin{equation}\label{eq:assump_ratio}
    \lim_{\lambda \to 0^+} \left(\frac{C_\lambda - c(H)}{\lambda}\right) = \zeta\in\mathbb{R},
\end{equation}
then there is a solution $u^{\eta,\zeta}_0$ of \eqref{eq:hj-0} such that $u_\lambda \to u^{\eta,\zeta}_0$ locally uniformly on $\Omega$ as $\lambda \to 0^+$ and 
\begin{equation*}
    u^{\eta,\zeta}_0 = \sup \mathcal{E}^{\eta,\zeta}. 
\end{equation*}
with 
\begin{equation}\label{eq:mathcalE}
    \begin{aligned}
        \mathcal{E}^{\eta,\zeta}
        &:= \Big\lbrace \om\text{ is a continuous viscosity subsolution of \eqref{eq:hj-0}}\; \Big|\\
        &   \qquad\int_{T\overline{\Omega}}\om(x)\partial_u L(x,v,0)d\mu+\eta\int_{T\overline{\Omega}}\langle\partial_xL(x,v,0),x\rangle d\mu + \zeta \geq 0,\;\forall \mu\in\mathcal{M}\Big\rbrace.
    \end{aligned}
\end{equation}
Conversely, if $\lim_{\lambda\to 0^+}u_\lambda = u_0$ locally uniformly for some $u_0\in\mathrm{C}(\overline{\Omega})$ then \eqref{eq:assump_ratio} holds with 
\begin{equation}
    \label{eq:u0-zeta}
    \zeta=\max_{\mu\in\cM} \Big\{-\int_{T\overline{\Omega}} u_0(y)\partial_u L(y,v,0)d\mu-\eta\int_{T\overline{\Omega}}\langle\partial_xL(y,v,0),y\rangle  d\mu\Big\}.
\end{equation}
\end{thm}

\begin{rmk}\label{rmk:case-c(lambda)} \quad 
\begin{itemize}
\item[(i)] An example was proposed in \cite[Theorem 1.8]{Tu}, which indicates that the solution $u_\lb$ of \eqref{eq:hj-lb} could diverge as $\lb\rightarrow 0^+$ if $\eta=-\infty$. Therefore, \ref{itm:L} is a necessary condition in the convergence of $u_\lb$.

\item[(ii)] Theorem \ref{thm:ge} covers two important cases: $C_\lambda = c(H)$ and $C_\lambda = c(\lambda)$ for $\lambda>0$, where $c(\lambda)$ is the ergodic constant of $H$ on $\Omega_\lambda$, which can be similarly given as in \eqref{eq:mane-cri} by:
    \begin{equation}\label{eq:mane-cri-lambda}
        c(\lambda) = \inf \left\lbrace c\in \R: H(x,Du(x),0) \leq c\;\text{has a solution in}\;\Omega_\lambda \right\rbrace .
    \end{equation}
    For discounted Hamiltonians, they correspond to two normalization $u_\lambda + \lambda^{-1}c(H)$ and $u_\lambda + \lambda^{-1}c(\lambda)$, where the convergence have been discussed in \cite{Tu}. 

\item[(iii)] If $C_\lambda = c(H)$ then \eqref{eq:assump_ratio} holds and $\xi = 0$, so the associated solution $u_\lambda \in \mathrm{C}(\overline{\Omega}_\lambda,\R)$ satisfies $\lim_{\lambda\to 0^+} u_\lambda = u^{\eta,0}_0 = \sup \mathcal{E}^{\eta, 0}$, with $\mathcal{E}^{\eta,0}=\cE^{\eta,\zeta}\big|_{\zeta=0}$  given by \eqref{eq:mathcalE}.

\item[(iv)] If $C_\lambda = c(\lambda)$ then $c(\lambda)\to c(H)$ as $\lambda \to 0^+$. If $\eta = 0$ then $\zeta$ has to be $0$ as well, since
\begin{equation}\label{eq:zeta}
    \frac{c(\lambda)-c(H)}{\lambda}  =  \left(\frac{r(\lambda)}{\lambda} \right)\left(\frac{c(\lambda)-c(H)}{r(\lambda)} \right)
\end{equation}
and (ii) of Proposition \ref{prop:aux-conv}. Consequently, $u_0^{\eta,\zeta}\equiv u_0$ as given in Theorem \ref{thm:fix}.

\item[(v)] If $C_\lb=c(\lb)$ and $\eta \neq 0$ then the existence of $\zeta\in\R$ in \eqref{eq:assump_ratio} is equivalent to the existence of 
\begin{equation}\label{eq:limit-r-lambda}
    \lim_{\lb\rightarrow 0^+}\frac{c(\lb)-c(H)}{r(\lb)}=\varsigma\in\R
\end{equation}
due to \eqref{eq:zeta}. A sufficient condition to guarantee that $\varsigma$ exists is 
\begin{equation}\label{eq:v-sig}
    \int_{T\ol\Om}\langle\partial_xL(x,v,0),x\rangle d\mu\equiv-\varsigma,\quad\forall \mu\in\cM, \tag{$\star$}
\end{equation}
which was firstly pointed out in \cite[Theorem 1.4]{Tu}. We note that $\varsigma$ is actually independent of the choice of $r(\lambda)$ and 
is predetermined by $L$ and $\mathcal{M}$. 
Moreover, once $\varsigma$ exists, $\zeta=\eta\cdot\varsigma$ can be deduced from \eqref{eq:zeta}. We refer the Remark \ref{rmk:gene-case} for 
a further discussion on \eqref{eq:limit-r-lambda}.

\item[(vi)] For $C_\lambda=c(\lambda)$, once $\zeta\in\R$ exists in \eqref{eq:assump_ratio} and if \eqref{eq:v-sig} holds, then using them in \eqref{eq:mathcalE} we can conclude that
$\mathcal{E}^{\eta,\zeta} \equiv \mathcal{E}$ as given in Theorem \ref{thm:fix}. Such a coercive consequence is new in the literature, and coincides with the previous 
results in \cite{WYZ}.

\item[(vii)] Inspired by \eqref{eq:u0-zeta}, 
we can deduce an alternative criterion for the existence of $\zeta$, as a generalization of \eqref{eq:v-sig} (see subsection \ref{ss} for more details).
\end{itemize}
\end{rmk}
As a matter of fact, if we strengthen assumptions \ref{itm:C4} and \ref{itm:C5} to the following:
\begin{description}[style=multiline, labelwidth=1.1cm, leftmargin=2cm]
    \item[\namedlabel{itm:C45'}{$(\mathcal{C}_{45}')$}] For each $R>0$, there exists a constant $C_R>0$ such that 
    \begin{equation*}
        \left\{
            \begin{aligned}
            |H(x,p, u)-H(y,p,u)|\leq C_R|x-y|,\\
            |H(x,p, u)-H(x,q,u)|\leq C_R|p-q|,
            \end{aligned}
        \right.
    \end{equation*}
for any $x,y\in\ol\cU$, $u\in\R$ and $p,q\in\R^n$ with $|u|,|p|,|q|\leq R$,
\end{description}
then the convergence of $u_\lambda$ in Theorem \ref{thm:ge} can be quickly proved by using merely the comparison principle, without resorting to \ref{itm:C6} and \ref{itm:C7}:

\begin{cor}[Easier convergence with $\eta=0$]\label{cor:eta=0} Let $\Omega$ satisfy \ref{itm:C0''}, $H$ satisfy \ref{itm:C1'}--\ref{itm:C3} and \ref{itm:C45'}, $r(\lambda)$ satisfy \ref{itm:L} with $\eta = 0$ and $\zeta = 0$ in \eqref{eq:assump_ratio}. Then the solution $u_\lambda \in \mathrm{C}(\overline{\Omega}_\lambda)$ of \eqref{eq:hj-lb} converges to $u_0$, the limiting solution as given in Theorem \ref{thm:fix}. 
\end{cor}

As a necessary complement, we can generalize previous conclusions to the case where $\lim_{\lb\rightarrow 0^+}C_\lb=c\in\R$ (it may happen that $c\neq c(H)$). 

\begin{thm}\label{thm:change0-domain-c-ener}
Suppose $\Omega$ satisfies \ref{itm:C0''}, $r(\lambda)$ satisfies \ref{itm:L} and the Hamiltonian $H$ satisfies \ref{itm:C1'}--\ref{itm:C7}. For $C_\lambda\in \mathbb{R}$ such that $C_\lambda\to c$ as $\lambda\to 0^+$, let $v_\lambda\in \mathrm{C}(\overline{\Omega}_\lambda)$ be the solution to \eqref{eq:hj-lb} with $C_\lambda \to c$ as $\lambda \to 0^+$. Then $\lambda v_\lambda(\cdot)\rightarrow h_0(c)$ as $\lambda\rightarrow 0^+$ for certain constant $h_0(c)\in\R$. Furthermore, we have 
\begin{equation*}
    \lim_{\lambda\to 0^+}\left(v_\lambda(\cdot) - \frac{h_0(c)}{\lambda} \right) = v_c^{\eta,\zeta}(\cdot) \;\text{exists} \qquad\text{iff}\qquad \lim_{\lambda \to 0^+} \left(\frac{C_\lambda-c}{\lambda}\right) = \zeta\in \mathbb{R}\;\text{exists}.
\end{equation*}
The convergence is locally uniformly on $\Omega$. In such a case, we have $v_c^{\eta,\zeta}=\sup\cE_c^{\eta,\zeta}$ with 
\begin{equation*}
    \begin{aligned}
    \cE_c^{\eta,\zeta}
    &:=\Big\{w\in \mathrm{C}(\Omega);H(x, Dw,h_0(c))\leq c\;\text{in}\;\Omega\;\text{in the sense of viscosity}, \\
    & \qquad\quad   \int_{T\ol\Om}w(x)\partial_u L(x,v,h_0(c))d\mu+\eta\int_{T\ol\Om}\langle\partial_xL(x,v,h_0(c)),x\rangle d\mu + \zeta \geq 0,\;\forall \mu\in\cM_c\Big\}    
    \end{aligned}
\end{equation*}
and $\mathcal{M}_c$ is the set of Mather measures associated with Hamiltonian $H(\cdot,\cdot, h_0(c)): T^*\Omega\rightarrow\mathbb{R}$. 
\end{thm}

\begin{rmk}\quad 
\begin{itemize}

    \item[$\mathrm{(i)}$] Throughout all above Theorems and Corollaries, \ref{itm:C1'} is assumed for using the comparison principle (see Theorem \ref{thm:comp-p}). Due to the recent work \cite{CFZZ}, it is possible to weaken \ref{itm:C1'} to the following
    \smallskip
    \begin{description}[style=multiline, labelwidth=1.1cm, leftmargin=2cm]
        \item[\namedlabel{itm:C1}{$(\mathcal{C}_1')$}] $H(x,p,u)$ is strictly increasing in $u$ for any $(x,p)\in T^*\overline{\mathcal{U}}$.
    \end{description} \smallskip
    or even weaker, but that relies on a more complicated trajectory analysis from the weak KAM theory.  For the consistency and concision of the current article, we only point out this fact. 
    
    \item[$\mathrm{(ii)}$] While we primarily focus on $\Omega_\lambda$ as an affine deformation of $\Omega$ in our main conclusions, we believe our proof ideas can naturally extend to more general $\Omega_\lambda$ cases, as suggested in \cite[Section VII]{C-DL}. In the literature, the general perturbation is often $\Omega_\lambda = \Phi_\lambda(\Omega)$, where $\Phi_\lambda(x) = x + \lambda\mathbf{V}(x)$ with $\mathbf{V}$ a smooth vector field on $\mathbb{R}^n$. Our paper deals with $\mathbf{V} = \mathbf{Id}$, leaving the treatment of the general case for future work.

    \smallskip
    
    \item[$\mathrm{(iii)}$] A subsequent issue following the above Theorems and Corollaries is the estimation of the convergence rate of the solutions as $\lambda$ approaches $0^+$:

    \begin{que}\label{que:rate}
    For the equation 
    \begin{equation*}
        \begin{cases}
            \begin{aligned}
                H(x,D  \vartheta_\lambda(x), \lambda\vartheta_\lambda(x))\leq c(H),\quad x\in \Omega_\lambda,\\
                H(x,D  \vartheta_\lambda(x), \lambda\vartheta_\lambda(x))\geq c(H),\quad x\in \overline{\Omega}_\lambda,
            \end{aligned}
        \end{cases}
    \end{equation*}
with $\Om_\lb\rightarrow\Om$ as $\lb\rightarrow 0^+$, in what rate with respect to $|r(\lb)|$ does the associated viscosity solution $\vartheta_\lb$ converges to $\vartheta_0^\eta$ as given in Theorem \ref{thm:ge}?
\end{que}
In \cite{KTT}, the convergence rate of solutions for discounted H-J equations on nested domains was estimated, but it requires a non-vanishing discount, which is not assured in Open Problem \ref{que:rate} due to $\lambda\vartheta_\lambda \rightarrow 0$. For a pendulum system, the convergence rate of solutions as the discount vanishes has been established in \cite{mitake_weak_2018}, offering insights into understanding Open Problem \ref{que:rate}.

\item[$\mathrm{(iv)}$] For general boundary value problems, the vanishing discount limit has been studied in \cite{IMT2}, where they used a duality representation for the Mather measures. As far as we know, that is the only representation of the Mather measures in the boundary value problems. The compatibility of minimal curves with the boundary restrictions (Neumann or Dirichlet type) brings more difficulties in defining Mather measures by using the time-average along the minimal curves, since mostly minimal curves exist only for finite time. Therefore, it is meaningful but needs more effort to generalize the definition of Mather measures to contact H-J equations with nonlinear $u-$dependence, under wider boundary conditions. 

\end{itemize}
\end{rmk}

\subsection{Alternative criterion for \eqref{eq:assump_ratio}} \label{ss} For generalized $C_\lb$ converging to $c(H)$ as $\lb\rightarrow 0_+$,  the associated \eqref{eq:assump_ratio} may not exist, therefore, we try to find a criterion for \eqref{eq:assump_ratio} just like we did in \eqref{eq:v-sig} for the special case $C_\lb=c(\lb)$. Such a criterion doesn't involve any prior assumption on $\lim_{\lb\rightarrow 0^+}u_\lb$, therefore should be more essential than \eqref{eq:u0-zeta}. As we have shown above, for any $c\in\R$, there exists a unique number  $h_\lambda(c)$ such that  the ergodic constant of the Hamiltonian $H\left(\cdot,\cdot, h_\lambda(c)\right)$ on $\Omega_\lambda$ is $c$. Due to \ref{itm:C1'}, \ref{itm:C2} and \ref{itm:C3}, we can see that $h_\lb(c)$ is continuous and strictly increasing in $c$, and $\lim_{\lb\rightarrow 0^+}h_\lb(c)=h_0(c)$ (see Section \ref{sec:4-generalized-conv-c} for the proof). In fact, the inverse function $h_\lb^{-1}: a\in\R\rightarrow\R$ is exactly the ergodic constant of $H(\cdot,\cdot, a)$ on $\Om_\lb$. With these preparations, we can give a different criterion for $\zeta$ in \eqref{eq:assump_ratio}:
\begin{thm}\label{thm:characterization-zeta} 
Let $\Omega$ satisfy \ref{itm:C0''}, $r(\lambda)$ satisfy \ref{itm:L} and the Hamiltonian $H$ satisfy \ref{itm:C1'}--\ref{itm:C7}. If $C_\lambda \to c(H)$ as $\lambda\to 0^+$ then there exists $\nu\in \mathcal{M}$ such that 
\begin{equation*}
    \lim_{\lambda\to 0^+}\frac{1}{\lambda}
    \left( 
    C_\lambda - c(H) + h_\lambda(C_\lambda) \int_{T\overline{\Omega}} \partial_u L(x,v,0)\;d\nu
    \right) = \eta \int_{T\overline{\Omega}} \langle \partial_x L(x,v,0),-x\rangle\;d\nu.
\end{equation*}
Consequently, $\zeta = \lim_{\lambda\to 0^+}(C_\lambda-c(H))/\lambda$ exists if and only if $\rho = \lim_{\lambda \to 0^+} h_{\lambda}(C_\lambda)/\lambda$ exists. Once any one of $\zeta$ and $\rho$ exists, the other one can be identified by the following relation:
\begin{equation}\label{eq:characterization-zeta}
    \zeta = \max_{\mu\in \mathcal{M}} \left\lbrace -\rho\int_{T\overline{\Omega}} \partial_u L(x,v,0)\;d\mu -  \eta \int_{T\overline{\Omega}} \partial_x \langle L(x,v,0), x\rangle\;d\mu \right\rbrace.
\end{equation}
\end{thm}
 In Remark \ref{rmk:gene-case}, we will see how the measure $\nu\in\cM$ can be obtained. Here we propose the following corollaries of Theorem \ref{thm:characterization-zeta}:
\begin{cor}  \label{cor:diff-h0(c)-2}
Let $\Omega$ satisfy \ref{itm:C0''}, Hamiltonian $H$ satisfy \ref{itm:C1'}, \ref{itm:C2}, \ref{itm:C3}, \ref{itm:C6} and $c$ be in a neighborhood of $c(H)$. Then the followings are equivalent:
\begin{itemize}
    \item[(i)]  $c\mapsto h_0(c)$ is differentiable at $c = c(H)$;
    \item[(ii)] $-\int_{T\overline{\Omega}} \partial_uL(x,v,0)\;d\mu(x,v) = {\rm const} $ for all $\mu\in \mathcal{M}$.
\end{itemize}
If (ii) holds, then the constant in (ii) is forced to be  $1/h'_0(c(H))$.
\end{cor}
\begin{cor}\label{cor:characterization-zeta-limit-exists}
Let $\Omega_\lb$ satisfy \ref{itm:C0''}, $r(\lambda)$ satisfy \ref{itm:L} and the Hamiltonian $H$ satisfy \ref{itm:C1'}--\ref{itm:C7}.  For any $c\in\R$, if  $C_\lambda\to c$ as $\lambda\to 0^+$, then there exists $\nu\in \mathcal{M}_c$ such that
\begin{equation*}
    \lim_{\lambda\to 0^+}
    \left( 
    \frac{C_\lambda - c}{\lambda} + \frac{h_\lambda(C_\lambda) - h_0(c)}{\lambda}\int_{T\overline{\Omega}} \partial_u L(x,v,h_0(c) )d\nu
    \right) = \eta \int_{T\overline{\Omega}} \langle \partial_x L(x,v,h_0(c) ),-x\rangle d\nu.
\end{equation*}
\end{cor}

\subsection{Organization of the paper} In Section \ref{sec:2-preliminaries}, we discuss the existence and well-posedness of solutions to \eqref{eq:hj-lb} introducing the main tool: the variational principle for contact Hamiltonians with state constraints. We also define Mather measures, explore their properties, and prove Theorem \ref{thm:fix} for a fixed domain. Section \ref{sec:3-conv-variable-domains-general} is devoted to prove Theorem \ref{thm:ge}. In Section \ref{sec:4-generalized-conv-c} we consider the convergence of solutions associated with variable domains and arbitrary $\lim_{\lambda \rightarrow 0^+}C_\lambda=c\in\mathbb{R}$ in \eqref{eq:hj-lb}, and the alternative criterion of \eqref{eq:assump_ratio}. For readability, the proofs of some technical lemmas are moved to the Appendix.

\section{Variational principle of the contact Hamilton--Jacobi equation with state-constraint}\label{sec:2-preliminaries}

As is shown in \eqref{eq:mane-cri}, $c(H)$ is finite and unique. Besides, \eqref{eq:hj-0} always possesses a viscosity solution (due to a vanishing discount approach in \cite{Tu}), which is Lipschitz in $\Omega$. Without loss of generality, we assume $u$ is such a solution of \eqref{eq:hj-0}, then 
\begin{equation*}
    \begin{cases}
        u^-:=u-\Vert u\Vert_{L^\infty(\Omega)}\leq 0,\\
        u^+:=u+\Vert u\Vert_{L^\infty(\Omega)}\geq 0
    \end{cases}, \qquad x\in \Omega
\end{equation*}
present as a subsolution in $\Om$ and a supersolution on $\ol\Om$, respectively, of the following 
\begin{equation}\label{eq:hj-fix}
    H(x,D  u(x),u(x))=c(H)    
\end{equation}
due to \ref{itm:C1'} ($u\mapsto H(x,p,u)$ is only required to be non-decreasing for this conclusion). Recall that viscosity subsolutions of \eqref{eq:hj-fix} in $\cU$ are Lipschitz, and therefore they are equivalent to a.e. subsolutions (see \cite{BJ,tran_hamilton-jacobi_2021}). 

\begin{thm}[Perron's method \cite{KTT}] \label{thm:perron}
Suppose
\begin{equation*}
    C_l:=\max\left\{|p|: H(x,p,u^-(x))\leq c(H), \;x\in\overline{\mathcal{U}}\right\}    
\end{equation*}
and 
\begin{equation*}
    \cF:=\Big\{\text{$\om\in C(\ol\Om)$ is a subsolution of\;\eqref{eq:hj-fix}\; in $\Om$}\;
    \Big|\; u^-\leq\om\leq u^+, \Vert D \om\Vert_{L^\infty(\Om,\R)}\leq C_l\Big\}, 
\end{equation*}
then $u^*:=\sup\cF$ is a viscosity solution of 
\beq\label{eq:hj-fix-sc}
\quad\left\{
\begin{aligned}
H\left(x,D  \vartheta(x),\vartheta(x)\right)\leq c(H),\quad x\in \Om,\\
H\left(x,D  \vartheta(x),\vartheta(x)\right)\geq c(H),\quad x\in \ol\Om.\\
\end{aligned}
\right.
\eeq
\end{thm}

\begin{proof} Because of the Lipschitz property of functions in $\mathcal{F}$, $u^* \in \mathrm{C}(\overline{\Omega})$. Clearly, $u^*$ is a subsolution of \eqref{eq:hj-fix} in $\Om$. If it is not a supersolution on $\ol\Om$, then there exists a $x_0\in\ol\Om$ and a function $\varphi\in C^1(\ol\Om)$ satisfying $u^*(x_0)=\varphi(x_0)$, 
\begin{equation*}
	\left\{
		\begin{aligned}
			&\Vert D\varphi\Vert _{L^\infty(B(x_0,r)\cap\ol\Om)}\leq C_l, \\
	 		&u^*(x)\geq \varphi(x)+|x-x_0|^2,\quad\forall x\in B(x_0,r)	\cap\ol\Om
		 \end{aligned}
	 \right.
\end{equation*}
for some $r>0$, and 
\begin{equation}\label{eq:<}
	H(x_0,D \varphi(x_0),\varphi(x_0))<c(H).
\end{equation}
By definition, $u^*\leq u^+$ on $\overline{\Omega}$, therefore $\varphi(x_0) = u^*(x_0) \leq u^+(x_0)$. 
If $\varphi(x_0)=u^+(x_0)$, then
\begin{equation*}
    H\left(x_0,D \varphi(x_0), u^+(x_0)\right)=H(x_0,D \varphi(x_0), \varphi(x_0)\geq c(H),
\end{equation*}
which contradicts with \eqref{eq:<}. So $\varphi(x_0)<u^+(x_0)$. Due to the continuity of $\varphi$ and $H$, we can find $0<\eps<r/2$ sufficiently small such that 
\begin{equation*}
    \begin{cases}
        \varphi(x)+\eps^2<u^+(x)\\
        H(x,D \varphi(x),\varphi(x)+\eps^2)<c(H)
    \end{cases}\qquad \text{for}\;x\in B(x_0,2\eps)\cap\ol\Om.
\end{equation*}
That implies $\varphi(x)+\eps^2$ is a strict subsolution in $B(x_0,2\eps)\cap\ol\Om$. If we define 
\[
\om(x):=\left\{
\begin{aligned}
&\max\{u^*(x), \varphi(x)+\eps^2\}, \quad &x\in B(x_0,\eps)\cap\ol\Om, \\
& u^*(x),\quad &x\in \ol\Om\backslash B(x_0,\eps),
\end{aligned}
\right.
\]
then it has to be a subsolution of \eqref{eq:hj-fix} and is contained in $\cF$. However, 
\[
\om(x_0)\geq \varphi(x_0)+\eps^2=u^*(x_0)+\eps^2>u^*(x_0)=\sup\cF(x_0)
\]
contradicts with the definition of $\cF$, so $u^*$ is also a supersolution of \eqref{eq:hj-fix} on $\ol\Om$.
\end{proof}

\begin{rmk}\label{rmk:superlinear} As a consequence of the gradient bound $C_l$, we see that the value of $H(x,p,u)$ for $|p|$ large is irrelevant when considering solutions to such a state-constraint problem \eqref{eq:hj-lb}. Therefore, one can modify $H(x,p,u)$ when $|p|$ large enough to strenghthen \ref{itm:C3} to \ref{itm:C3'} as follows
\begin{description}[style=multiline, labelwidth=1.1cm, leftmargin=2cm]
    \item[\namedlabel{itm:C3'}{$(\mathcal{C}_3')$}] $\lim_{|p|\rightarrow+\infty}H(x,p,u)/|p|=+\infty$ for all $(x,u)\in\ol\cU\times\R$.
\end{description}
That is why we can always assume the Lagrangian is finite.
\end{rmk}

On the other side, the comparison principle guarantees that $u^*$ is the unique solution of \eqref{eq:hj-fix-sc}. For consistency, the comparison principle (following \cite{C-DL}) is independently supplied here (see the proof in Appendix). We can refer to \cite{CIL} for other versions of the comparison principle. 

\begin{thm}[Comparison principle]\label{thm:comp-p} Let $\Omega$ satisfy \ref{itm:C0''} and $H$ satisfy \ref{itm:C1'}, \ref{itm:C3}, \ref{itm:C4}. If $v_1\in \mathrm{C}(\overline{\Omega})$ be a subsolution of \eqref{eq:hj-fix} in $\Omega$ and $v_2\in \mathrm{C}(\overline{\Omega})$ is a supersolution of \eqref{eq:hj-fix} on $\overline{\Omega}$, then $v_1\leq v_2$ on $\overline{\Omega}$. 
\end{thm}

Next, we propose an alternative interpretation to $u^*$ from the viewpoint of optimal control theory:

\begin{prop}\label{prop:opt-formula}
For any Hamiltonian $H$ satisfying \ref{itm:C1'}, \ref{itm:C2}, \ref{itm:C3'}, there holds
\beq\label{eq:sol-exp}
u^*(x)=\inf_{\substack{\gamma\in {\rm Lip}([t,0],\ol\Om)\\\gamma(t)=y,\gamma(0)=x}}\Big\{u^*(\gamma(t))+\int_t^0\Big( L\big(\gamma(s),\dot\gamma(s),u^*(\gamma(s))\big)+c(H)\Big) ds\Big\}
\eeq
for any $t\leq 0$. Moreover, there exists a uniformly Lipschitz curve $\gamma_x:(-\infty,0]\rightarrow\ol\Om$ ending with $x$, such that
\beq
u^*(x)-u^*(\gamma_x(t))=\int_t^0\Big( L\big(\gamma_x(s),\dot\gamma_x(s),u^*(\gamma_x(s))\big)+c(H)\Big) ds,\quad\forall\, t\leq 0.
\eeq
\end{prop}
\begin{proof}
As $u^*(x)\in C(\ol\Om)$ is time-independent, $u(x,t) = u^*(x)$ is also the viscosity solution of the Cauchy problem 
\begin{equation}\label{eq:cauchy-sc}
    \begin{cases}
        \begin{aligned}
            \partial_t u(x,t)+H(x,D u, u)& \leq c(H), &\quad &(x,t)\in \Omega\times (0,T),\\
            \partial_t u(x,t)+H(x,Du, u)& \geq c(H), &\quad &(x,t) \in \overline{\Omega}\times (0,T),\\
            u(x,0)&=u^*(x), 
            &\quad &(x,t)\in \overline{\Omega}\times \{0\},
        \end{aligned}
    \end{cases}
\end{equation}
for any fixed $T\geq 0$. Due to \cite[Theorem X.1, (5)]{C-DL}, the viscosity solution of \eqref{eq:cauchy-sc} is unique and can be expressed by (see also \cite{Mitake2008})
\begin{equation*}
    U(x,t)=\inf_{\substack{\gamma\in {\rm Lip}([0,t],\ol\Om)\\\gamma(t)=x}}\left\{u^*(\gamma(0))+\int_0^t\Big(L\big(\gamma(s),\dot\gamma(s),u^*(\gamma(s))\big)+c(H) \Big)ds\right\}
\end{equation*}
for all $t\in [0,T]$. Consequently, for any $x\in\ol\Om$ and $t\geq 0$,
\begin{equation*}
    u^*(x)= \inf_{\substack{\gamma\in {\rm Lip}([0,t],\ol\Om)\\\gamma(t)=x}}
    \left\{
        u^*(\gamma(0))+\int_0^t\Big(L\big(\gamma(s),\dot\gamma(s),u^*(\gamma(s))\big)+c(H) \Big)ds
    \right\}.
\end{equation*}
Recall that $u^*$ is bounded, $\ol\Om\subset\R^n$ is compact and $H$ is superlinear in $p$, we can always find a sequence of Lipschitz curves $\gamma_n:[0,t]\rightarrow \ol\Om$ with $\gamma_n(t)=x$, such that 
\begin{equation*}
    u^*(x)=\lim_{n\rightarrow +\infty}
    \left[u^*(\gamma_n(0))+\int_0^t\Big( L\big(\gamma_n(s),\dot\gamma_n(s),u^*(\gamma_n(s))\big)+c(H)\Big) ds\right].    
\end{equation*}
That implies 
\begin{equation*}
    \sup_{n\in\mathbb N} \left[ \int_0^t\Big( L\big(\gamma(s),\dot\gamma(s),u^*(\gamma(s))\big)+c(H)\Big) ds\right] \leq 2\Vert u^*\Vert_{L^\infty(\overline{\Omega})}<+\infty.
\end{equation*}
Due to the {\it Dunford-Pettis Theorem} (see \cite[Theorems 2.1,\,2.2 and 3.6]{Buttazzo1998OnedimensionalVP} for instance), there must exist a subsequence $\gamma_{n_k}$ uniformly converging to an absolutely continuous curve $\gamma:[0,t]\rightarrow \ol\Om$, such that 
\begin{equation*}
\begin{aligned}
    & \int_0^t\Big( L\big(\gamma(s),\dot\gamma(s),u^*(\gamma(s))\big)+c(H)\Big) ds\\
    &\qquad\qquad \leq \liminf_{k\rightarrow+\infty}\left[ \int_0^t\Big( L\big(\gamma_{n_k}(s),\dot\gamma_{n_k}(s),u^*(\gamma_{n_k}(s))\big)+c(H) \Big) ds\right].
\end{aligned}
\end{equation*}
This readily implies 
\begin{equation*}
    u^*(x)-u^*(\gamma(0))=\int_0^t
    \Big( 
        L\big(\gamma(s),\dot\gamma(s),u^*(\gamma(s))\big)+c(H)
    \Big) ds.
\end{equation*}
Since $t\geq 0$ is freely chosen above, that implies for any $n\in\mathbb N$ we can find an absolutely continuous curve $\xi_n:[-n,0]\rightarrow\ol\Om$ with $\xi_n(0)=x$, such that 
\begin{equation*}
    u^*(\xi_n(b))-u^*(\xi_n(a))=\int_{a}^b 
    \Big( 
    L\big(\xi_n(s),\dot\xi_n(s),u^*(\xi_n(s))\big)+c(H)
    \Big) ds    
\end{equation*}
for any $-n\leq a\leq b\leq 0$. By a diagonal argument, there exists an absolutely continuous curve $\xi_x : (-\infty, 0] \rightarrow\ol\Om$ with $\xi_x(0) = x$, which is the uniform limit of the curves $\xi_n$ (up to extraction of a subsequence) over any compact subinterval $[a,b]\subset(-\infty,0]$. We now show that $|\dot{\xi}_x(s)|\leq M$ for all $s\in (-\infty, 0]$. Since 
\begin{equation*}
    u^*(\xi_x(b))-u^*(\xi_x(a))=\int_{a}^b\Big( L\big(\xi_x(s),\dot\xi_x(s),u^*(\xi_x(s))\big)+c(H) \Big) ds    
\end{equation*}
for any $a\leq b\leq 0$ and 
\[
L(x,v,u^*(x))+c(H)\geq (C_l+1)|v|-M,\quad\forall (x,v)\in T\ol\Om
\]
for certain constant $M>0$ due to \ref{itm:C3'}, there should be 
\ben
u^*(\xi_x(b))-u^*(\xi_x(a))&\geq &\int_a^b \Big((C_l+1)|\dot\xi_x(s)|-M\Big) ds\\
&\geq &C_l{\rm dist}(\xi_x(a),\xi_x(b))+ \int_a^b |\dot\xi_x(s)| ds-M(b-a).
\een
On the other side, $\Vert D  u^*\Vert _{L^\infty(\ol\Om)}\leq C_l$ due to Theorem \ref{thm:perron}, so we get 
\[
\frac1{b-a}\int_a^b  |\dot\xi_x(s)| ds\leq M,\quad\forall a<b\leq 0.
\]
Taking $b-a\rightarrow 0^+$ we infer that $\xi_x:(-\infty,0]\rightarrow\ol\Om$ is uniformly Lipschitz, i.e.
\begin{equation*}
    |\dot\xi_x(a)|\leq M,\quad{\rm a.e.}\; a\in(-\infty,0].
\end{equation*}
So we finish the proof.
\end{proof}

\begin{defn}
A probability measure $\mu\in\cP (T\ol\Om,\R)$ is called {\bf holonomic} if it satisfies:
\begin{itemize}
    \item[(i)] $\int_{T\ol\Om}|v|d\mu(x,v)<+\infty$;
    \item[(ii)] $\int_{T\ol\Om}\langle D\phi(x),v\rangle d\mu(x,v)=0$ for every $\phi\in \mathrm{C}^1(\ol\Om)$.
\end{itemize}
We denote by $\cC(T\ol\Om,\R)$ the set of all holonomic measures.
\end{defn}

\begin{lem} Assume \ref{itm:C2}, \ref{itm:C3'}, \ref{itm:C4} and $\Omega$ satisfies \ref{itm:C0''}, then 
\begin{equation*}
    \int_{T\overline{\Omega}} L(x,v,0)d\mu \geq -c(H) \qquad\text{for all}\;\mu \in \mathcal{C}(\overline{\Omega}, \R).
\end{equation*}
\end{lem}
\begin{proof} From the assumptions we have $L(x,v,0)$ is finite and there exists a subsolution $w\in \mathrm{C}(\overline{\Omega})$ to $H(x,Dw(x), 0) \leq c(H)$ in $\Omega$. Using Lemma \ref{lem:smooth_smaller_domain} below, for every $\delta> 0$ we can find $\widehat{w}_\delta\in \mathrm{C}^1(\overline{\Omega})$ such that $\widehat{w}_\delta \to w$ a.e. as $\delta \to 0^+$ and
\begin{equation*}
    H(x,D\widehat{w}_\delta(x), 0) \leq c(H) + \mathcal{O}(\delta) \qquad\text{in}\;\Omega.
\end{equation*}
If $\mu \in \mathcal{C}(\overline{\Omega}, \R)$ then for $(x,v)\in T\overline{\Omega}$ we have
\begin{equation*}
    c(H) + \mathcal{O}(\delta) + L(x,v, 0) \geq H\left(x,D\widehat{w}_\delta(x),0\right) + L(x,v,0) \geq \langle D\widehat{w}_\delta, v\rangle.
\end{equation*}
Taking integration and using the definition of holonomic measure we have
\begin{equation*}
    c(H) +  \mathcal{O}(\delta) +\int_{T\overline{\Omega}} L(x,v,0)\;d\mu \geq 0.
\end{equation*}
Let $\delta\to 0^+$ we obtain the conclusion.
\end{proof}

\begin{defn}[Mather measure]\label{defn:Mather-measures} A holonomic measure minimizing 
\begin{equation}\label{eq:defn-Mather}
    \min_{\mu\in \cC(T\ol\Om,\R)}
    \int_{T\ol\Om}L(x,v,0)\,d\mu=-c(H).   
\end{equation}
is called a {\bf Mather measure}. We denote by $\mathcal{M}$ the set of all Mather measures in $\mathcal{C}(T\overline{\Omega},\R)$ associated with $H(x,p,0)$ on $\Omega$.
\end{defn}

We now give a proof of Theorem \ref{thm:fix} for the convergence of solutions on a fixed domain:

\begin{proof}[Proof of Theorem \ref{thm:fix}] Due to Theorem \ref{thm:perron}, $\{u_\lb\}_{\lb>0}$ are equi-bounded and equi-Lipschitz. By {\it Arzel\`a--Ascoli Theorem}, the accumulating functions of $\{u_\lb\}_{\lb>0}$ in $\mathrm{C}(\overline{\Omega})$ as $\lb\rightarrow 0^+$ always exist, so it remains to prove that the set of accumulating functions is a singleton. \smallskip

Suppose that $\vartheta_0 = \lim_{i\to  +\infty}\vartheta_{\lambda_i}$ is a subsequence limit, we show that $\vartheta_0\in \mathcal{E}$. It is clear that $\vartheta_0$ is a subsolution to \eqref{eq:hj-0} by stability of viscosity solutions. We proceed to show that 
\begin{equation}\label{eq:showing-cond-S}
    \int_{T\overline{\Omega}} \partial_u L(x,v,0)\vartheta_0(x)d\mu(x,v)\geq 0 \qquad\text{for all}\;\mu\in \mathcal{M}.
\end{equation}
Using the subsequence $i\to +\infty$, we have
\begin{align}
    0   &\leq\frac{1}{\lambda_i} \int_{T\ol\Om} \Big(L(x,v,\lambda_i \vartheta_{\lambda_i}(x)) +c(H)\Big) d\mu(x,v)\nonumber \\
        &=\frac1{\lambda_i}\int_{T\ol\Om}\Big( L(x,v,\lambda_i \vartheta_{\lambda_i}(x)) -L(x,v,0)+L(x,v,0)+c(H)\Big)d\mu(x,v) \nonumber \\
        &=\int_{T\ol\Om} \left(\frac{L(x,v,\lambda_i \vartheta_{\lambda_i}(x)) -L(x,v,0)}{\lambda_i}\right) d\mu(x,v)\nonumber 
\end{align}
where we use $\mu\in \mathcal{M}$ for the cancellation of the last terms. Due to \ref{itm:C6}, taking $\lambda_i\rightarrow 0$ in above inequality we obtain \eqref{eq:showing-cond-S}, then $\vartheta_0\in \mathcal{E}$. \smallskip

Next, we show that $\vartheta_0=\sup \mathcal{E}$. Suppose $w\in \cE$ is a viscosity subsolution of \eqref{eq:hj-0}, i.e., $H(x, Dw(x), 0) \leq c(H)$ in $\Omega$, we show that $w\leq \vartheta_0$ in $\Omega$. Using Lemma \ref{lem:smooth_smaller_domain}, for every $\delta > 0$ there exists $\widehat{w}_\delta\in \mathrm{C}^1(\overline{\Omega})$ such that $\widehat{w}_\delta \to w$ a.e. as $\delta \to 0^+$ and
\begin{equation}\label{-2}
    H(x,D\widehat{w}_\delta(x), 0) \leq c(H) + \mathcal{O}(\delta) \qquad\text{in}\;\Omega.
\end{equation}
Let $x\in \Omega$, thanks to Proposition \ref{prop:opt-formula}, we can find a Lipschitz curve $\gamma_x^\lb:(-\infty, 0]\rightarrow \ol\Omega$ ending with $\gamma_x^\lb(0) = x$ such that
\begin{align}\label{eq:d/dt-u-lambda}
    \frac{\mbox{d}}{\mbox{d}t}\left(\vartheta_{\lambda}(\gamma_x^\lb(t))\right)
    &=L\left(\gamma_x^\lb(t),\dot\gamma_x^\lb(t),\lb \vartheta_{\lambda}(\gamma_x^\lb(t))\right)+c(H)\\
    &=L\left(\gamma_{x}^\lb(t),\dot\gamma_{x}^\lb(t),0\right) +c(H)+\lb \vartheta_{\lambda}\big(\gamma_{x}^\lb(t)\big)\partial_u L\left(\gamma_{x}^\lb(t),\dot\gamma_{x}^\lb(t),0\right)+\Theta_x^\lb(t) \nonumber
\end{align}
for a.e. $t\in(-\infty,0]$, where $\frac{\mbox{d}}{\mbox{dt}}\gamma^\lambda_x(t) = \dot{\gamma}^\lambda_x(t)$ and
\begin{equation*}
\begin{aligned}
    \Theta_x^\lb(t):=
    L\left(\gamma_x^\lb(t),\dot\gamma_x^\lb(t),\lb \vartheta_{\lambda}\big(\gamma_x^\lb(t)\big)\right) &-L\left(\gamma_{x}^\lb(t),\dot\gamma_{x}^\lambda(t),0\right)\\
    &-\lb \vartheta_{\lambda}\big(\gamma_{x}^\lambda(t)\big)\, \partial_u L\left(\gamma_{x}^\lb(t),\dot\gamma_{x}^\lambda(t),0\right).
\end{aligned}
\end{equation*}
Let us define
\begin{equation}\label{eq:def_alpha(t)}
    \alpha^\lambda_x(t):=\int_0^t\partial_u L\left(\gamma_{x}^\lb(s),\dot\gamma_{x}^\lb(s),0\right)ds, \qquad t\in (-\infty, 0]
\end{equation}
then for all $t\in (-\infty, 0]$ there holds
\begin{equation*}
    \frac{\mbox{d}}{\mbox{d}t} \left(\alpha^\lambda_x(t)\right) = \partial_u L\left(\gamma_{x}^\lb(t),\dot\gamma_{x}^\lb(t),0\right), \qquad\text{and}\qquad \lim_{t\to -\infty} \alpha_x^\lambda(t) = +\infty
\end{equation*}
and it grows at least linearly due the the fact that $|\dot{\gamma}^\lambda_x(\cdot)|$ is uniformly bounded and \ref{itm:C1'}. We observe that $\Theta^\lambda_x(t)/\lambda \to 0$ uniformly in $t\leq 0$ as $\lb\rightarrow 0^+$ due to \ref{itm:C6}, therefore
\begin{equation}\label{eq:vanish_Theta}
    \lim_{\lambda \to 0^+}\int_{-\infty}^0 \Theta_x^\lb(t) e^{-\lb\alpha_x^\lb(t)} dt\rightarrow 0
\end{equation}
thanks to Lemma \ref{lem:est_alpha(t)}. From the definition of Legendre's transform and the smoothness of $\widehat{w}_\delta$ we have
\begin{align*}
    H\left(
        \gamma^\lambda_x(t),
        D\widehat{w}_\delta\big(\gamma^\lambda_x(t)\big),
        0
    \right) 
    + 
    L\left(
        \gamma^\lambda_x(t),
        \dot{\gamma}^\lambda_x(t),
        0
    \right)
    \geq 
    \left \langle \dot{\gamma}^\lambda_x(t), D\widehat{w}_\delta\big(\gamma^\lambda_x(t)\big)\right\rangle = \frac{\mbox{d}}{\mbox{d}t}
    \left(
        \widehat{w}_\delta(\gamma^\lb_{x}(t))
    \right)
\end{align*}
for all $t\in (\infty, 0]$. Therefore
\begin{align*}
    \frac{\mbox{d}}{\mbox{d}t}
    \left(
        \widehat{w}_\delta(\gamma^\lb_{x}(t))
    \right)
    \leq
    \frac{\mbox{d}}{\mbox{d}t}
    \left(
        u_{\lambda}(\gamma^\lb_{x}(t))
    \right)
    -\lb u_{\lambda}(\gamma_{x}^\lb(t))\partial_u L\left(\gamma_{x}^\lb(t),\dot\gamma_{x}^\lb(t),0\right)
    -\Theta_x^\lb(t) + \mathcal{O}(\delta)
\end{align*}
for a.e. $t\in (-\infty, 0]$. In other words, we have
\begin{equation*}
    \frac{\mbox{d}}{\mbox{d}t}
        \left(\vartheta_\lambda(\gamma^\lambda_x(t))\right) 
        - 
        \lambda \vartheta_\lambda(\gamma^\lambda_x(t)) 
        \frac{\mbox{d}}{\mbox{d}t}
            \left(\alpha^\lambda_x(t)\right)
        -\Theta^\lambda_x(t) 
        + \mathcal{O}(\delta) 
            \geq 
    \frac{\mbox{d}}{\mbox{d}t}
    \left(
        \widehat{w}_\delta(\gamma^\lb_{x}(t))
    \right)
\end{equation*}
and then 
\begin{align*}
    & e^{-\lambda \alpha^\lambda_x(t)}\frac{\mbox{d}}{\mbox{d}t}\left(\vartheta_\lambda(\gamma^\lambda_x(t))\right) - e^{-\lambda \alpha^\lambda_x(t)} \lambda \vartheta_\lambda\big(\gamma^\lambda_x(t)\big) \left(\frac{\mbox{d}}{\mbox{d}t}\left(\alpha^\lambda_x(t)\right)\right) \\
    &\qquad\qquad\qquad\qquad\qquad\qquad \geq e^{-\lambda \alpha^\lambda_x(t)}\left(\frac{\mbox{d}}{\mbox{d}t} \left(\widehat{w}_\delta(\gamma^\lambda_x(t))\right) + \Theta^\lambda_x(t) - \mathcal{O}(\delta)\right).
\end{align*}
Consequently, for a.e. $t\in (-\infty, 0]$, using the fact that $\alpha^\lambda_x(t)\to +\infty$ as $t\to -\infty$ we have
\begin{equation*}
    \frac{\mbox{d}}{\mbox{d}t}
    \Big(
        e^{-\lb\alpha_x^\lb(t)}\vartheta_\lambda(\gamma_x^\lb(t))
    \Big)
        \geq 
    e^{-\lb\alpha_x^\lb(t)}
    \left(
        \frac{\mbox{d}}{\mbox{d}t}
            \left(
                \widehat{w}_\dt(\gamma_x^\lb(t))
            \right)
        +\Theta_x^\lb(t)-\mathcal{O}(\delta)
    \right).
\end{equation*}
Integrating both sides with respect to $t\in(-\infty,0]$, we get 
\begin{equation}\label{eq:ineq}
    \vartheta_\lambda(x) 
        \geq 
    \widehat{w}_\dt(x)  
    -
    \int_{-\infty}^0
        \widehat{w}_\dt\big(\gamma_x^\lb(t)\big)
        \frac{\mbox{d}}{\mbox{d}t}
        \left(
            e^{-\lb\alpha_x^\lb(t)}
        \right) dt
    +
    \int_{-\infty}^0 
    \left(
        \Theta_x^\lb(t)-\mathcal{O}(\delta)
    \right)e^{-\lb\alpha_x^\lb(t)} dt.
\end{equation}
From Lemma \ref{lem:est_alpha(t)} there exists a positive constant $\mathcal{K}$ such that
\begin{equation}\label{eq:ratio-3}
 -\int_{-\infty}^0 e^{-\lambda \alpha^\lambda_x(t)}\mathcal{O}(\delta)\;dt \geq -\frac{\mathcal{O}(\delta)}{\mathcal{K}\lambda}.
\end{equation}
For $\mu \in \mathcal{M}$, we define $\mu_x^\lb\in\cP(T\ol\Omega,\R)$ by 
\begin{equation}\label{eq:new_measures}
    \displaystyle \int_{T\overline{\Omega}} f(y,v)d\mu_x^\lb(y,v):=
    \frac{
        \displaystyle \int_{-\infty}^0f(\gamma_x^\lb(t),\dot\gamma_x^\lb(t))e^{-\lb\alpha_x^\lb(t)}dt
    }
    {
        \displaystyle
        \int_{-\infty}^0e^{-\lb\alpha_x^\lb(t)}dt
    }, \qquad \text{for}\; f\in \mathrm{C}_c(T\ol\Omega).
\end{equation}
The second term in \eqref{eq:ineq} becomes
\begin{equation*}
-\int_{-\infty}^0w_\delta(\gamma_x^\lb(t))\frac{\mbox{d}}{\mbox{d}t}\left(e^{-\lb\alpha_x^\lb(t)}\right) dt 
= \lambda \left(\int_{-\infty}^0e^{-\lb\alpha_x^\lb(t)}dt\right) 
    \left(
        \int_{T\overline{\Omega}} \widehat{w}_\delta(y)
        \partial_uL
        \left(y,v,0\right)d\mu^\lambda_x(y,v)
    \right).
\end{equation*}
Observe that 
\begin{equation*}
\begin{aligned}
    \int_{T\overline{\Omega}}
    \partial_u L(y,v,0)d\mu_x^\lb(y,v) = \frac{\displaystyle\int_{-\infty}^0 \frac{\mbox{d}}{\mbox{d}t} \left(\alpha^\lambda_x(t)\right)e^{-\lambda \alpha^\lambda_x(t)}dt}{\displaystyle\int_{-\infty}^0 e^{-\lambda \alpha^\lambda_x(t)}dt}
    =\frac{1}{\displaystyle-\lb \int_{ -\infty}^0e^{-\lb\alpha_x^\lb(t)}dt}.
\end{aligned}
\end{equation*}
Therefore,
\begin{equation}\label{eq:ratio-1}
    -\int_{-\infty}^0
    \widehat{w}_\delta
    \big(\gamma_x^\lb(t)\big)
    \frac{\mbox{d}}{\mbox{d}t}
    \left(
        e^{-\lb\alpha_x^\lb(t)}
    \right) dt = 
    - \frac
    {
        \displaystyle\int_{T\overline{\Omega}} \widehat{w}_\delta(y)\partial_uL\left(y,v,0\right)d\mu^\lambda_x(y,v)
    }
    {
        \displaystyle\int_{T\overline{\Omega}} \partial_uL\left(y,v,0\right)d\mu^\lambda_x(y,v)
    }.
\end{equation}
Due to the following Lemma \ref{lem:mat-mea}, any accumulating measure  $\mu_x^0$ of $\{\mu_x^{\lb_n}\}_{n\in\N}$ (in the sense of weak topology) has to be contained in $\cM$, therefore along the sequence $\lambda_n$ we obtain from \eqref{eq:ratio-1} that
\begin{equation}\label{eq:ratio-2}
    \lim_{n\to \infty}  
    \frac
    {
        \displaystyle\int_{T\overline{\Omega}} \widehat{w}_\delta(y)\partial_uL\left(y,v,0\right)d\mu^{\lambda_n}_x(y,v)
    }
    {
        \displaystyle\int_{T\overline{\Omega}} \partial_uL\left(y,v,0\right)d\mu^{\lambda_n}_x(y,v)
    } = 
    \frac
    {
        \displaystyle\int_{T\overline{\Omega}} \widehat{w}_\delta(y)\partial_uL\left(y,v,0\right)d\mu^{0}_x(y,v)
    }
    {
        \displaystyle\int_{T\overline{\Omega}} \partial_uL\left(y,v,0\right)d\mu^{0}_x(y,v)
    }\leq 0
\end{equation}
where we use $\partial_uL(y,v,0) \leq 0$ and $\int_{T\overline{\Omega}} \widehat{w}_\delta(y)\partial_uL\left(y,v,0\right)d\mu^{\lambda_n}_x(y,v) \geq 0$ since $\mu^0_x \in \mathcal{M}$. Using \eqref{eq:ineq}, \eqref{eq:ratio-3} and \eqref{eq:ratio-1} we have
\begin{equation*}
     \vartheta_\lambda(x)  \geq \widehat{w}_\delta(x) + \int_{-\infty}^0 e^{-\lambda\alpha^\lambda_x(t)}\Theta^\lambda_x(t)\;dt - \frac{\mathcal{O}(\delta)}{\mathcal{K}\lambda} - \frac
    {
        \displaystyle\int_{T\overline{\Omega}} \widehat{w}_\delta(y)\partial_uL\left(y,v,0\right)d\mu^\lambda_x(y,v)
    }
    {
        \displaystyle\int_{T\overline{\Omega}} \partial_uL\left(y,v,0\right)d\mu^\lambda_x(y,v)
    }.
\end{equation*}
Taking $\delta \to 0^+$ we deduce that 
\begin{equation*}
     \vartheta_\lb(x)  \geq w(x)  + \int_{-\infty}^0 e^{-\lambda\alpha^\lambda_x(t)}\Theta^\lambda_x(t)\;dt  - \frac
    {
        \displaystyle\int_{T\overline{\Omega}} w(y)\partial_uL\left(y,v,0\right)d\mu^\lambda_x(y,v)
    }
    {
        \displaystyle\int_{T\overline{\Omega}} \partial_uL\left(y,v,0\right)d\mu^\lambda_x(y,v)
    } 
\end{equation*}
Taking $n\rightarrow+\infty$ for $\lambda_n$  and using \eqref{eq:vanish_Theta} and \eqref{eq:ratio-2},  we obtain
\begin{equation*}
    \vartheta_0(x)  \geq w(x) - \lim_{n\to \infty}  \frac
    {
        \displaystyle\int_{T\overline{\Omega}} w(y)\partial_uL\left(y,v,0\right)d\mu^{\lambda_n}_x(y,v)
    }
    {
        \displaystyle\int_{T\overline{\Omega}} \partial_uL\left(y,v,0\right)d\mu^{\lambda_n}_x(y,v)
    } \geq w(x).
\end{equation*}
Since $x\in \Omega$ is arbitrary, then $\vartheta_0(x)\geq w(x)$ in $\Omega$ and the proof is complete.    
\end{proof}

\begin{lem}\label{lem:mat-mea} In the context of Theorem \ref{thm:fix}, if $x\in \overline{\Omega}$ then any accumulating measure of $\mu_x^\lb$ as $\lb\rightarrow 0^+$ is a Mather measure. 
\end{lem}

\begin{proof} Suppose $\lim_{i\to +\infty}\mu_x^{\lambda_i}=\mu_x^0$ is an accumulating measure, we show that $\mu_x^0$ is holonomic and is a minimizing measure in the sense of Definition \ref{eq:defn-Mather}. Firstly, for any function $\phi\in C^1(\ol\Om)$, by definition as in equation \eqref{eq:new_measures} we have
\begin{align*}
    \int_{T\overline{\Omega}} \langle\nabla\phi,v\rangle d\mu_x^\lambda 
    &=
    \frac{
        \displaystyle
        \int_{-\infty}^0 \frac{\mathrm{d}}{\mathrm{d}t} \phi\big(\gamma^\lambda_x(t)\big) e^{-\lambda \alpha^\lambda_x(t)}dt
    }
    {
        \displaystyle
        \int_{-\infty}^0 e^{-\lambda \alpha^\lambda_x(t)}dt.
    } 
    = \frac{
        \displaystyle
        \phi\big(\gamma^\lambda_x(t)\big)e^{-\lambda \alpha^\lambda_x(t)}\Big|_{-\infty}^0 - 
        \int_{-\infty}^0  \phi\big(\gamma^\lambda_x(t)\big) \frac{\mathrm{d}}{\mathrm{d}t}\left(e^{-\lambda \alpha^\lambda_x(t)}\right)dt
    }
    {
        \displaystyle
        \int_{-\infty}^0 e^{-\lambda \alpha^\lambda_x(t)}dt.
    } \\
    &= \frac{
        \displaystyle
        \phi(x)- 
        \int_{-\infty}^0  \phi\big(\gamma^\lambda_x(t)\big) \frac{\mbox{d}}{\mbox{d}t}\left(e^{-\lambda \alpha^\lambda_x(t)}\right)dt
    }
    {
        \displaystyle
        \int_{-\infty}^0 e^{-\lambda \alpha^\lambda_x(t)}dt.
    } \\
    &\leq \lambda\mathcal{K} \Vert \phi\Vert_{L^\infty(\overline{\Omega})} \left(
        1+\int_{-\infty}^0 \frac{\mbox{d}}{\mbox{d}t}\left(e^{-\lambda \alpha^\lambda_x(t)}\right)dt
    \right) = 2\lambda\mathcal{K} \Vert \phi\Vert_{L^\infty(\overline{\Omega})}
\end{align*}
due to Lemma \ref{lem:est_alpha(t)}. Let $\lambda=\lambda_i$ and $i\rightarrow+\infty$ we have $\int_{T\overline{\Omega}} \langle\nabla\phi,v\rangle d\mu_x^0 = 0$ and thus $\mu_x^0$ is holonomic. Secondly, we have
\begin{equation*}
\begin{aligned}
    &\int_{T\overline{\Omega}}\Big(L(y,v,\lambda \vartheta_\lambda(y))+c(H)\Big) d\mu_x^\lb= 
    \frac{ 
        \displaystyle
        \int_{-\infty}^0\left[L\left(\gamma_x^\lambda(t),\dot\gamma_x^\lambda(t),\lambda \vartheta_\lambda(\gamma_x^\lambda(t))\right)+c(H)\right]e^{-\lb\alpha_x^\lambda(t)}dt
    }
    {
        \displaystyle
        \int_{-\infty}^0e^{-\lambda\alpha_x^\lambda(t)}dt
    }\\
    & \qquad\qquad  = 
    \frac
    {
        \displaystyle 
        \int_{-\infty}^0 \frac{\mbox{d}}{\mbox{d}t}\left(\vartheta_\lambda\big(\gamma_x^\lambda(t)\big)\right)e^{-\lb\alpha_x^\lb(t)} dt
    }
    {
        \displaystyle 
        \int_{-\infty}^0 e^{-\lb\alpha_x^\lb(t)}\;dt
    }\\
    &\qquad\qquad = 
    \frac
    {
        \displaystyle 
        \vartheta_\lambda\big(\gamma^\lambda_x(t)\big)e^{-\lambda \alpha^\lambda_x(t)}\Big|_{-\infty}^0 - \int_{-\infty}^0 \vartheta_\lambda\big(\gamma^\lambda_x(t)\big) \left(\frac{\mbox{d}}{\mbox{d}t} e^{-\lambda \alpha^\lambda_x(t)}\right)dt
    }
    {
        \displaystyle 
        \int_{-\infty}^0 e^{-\lb\alpha_x^\lb(t)}\;dt
    } \leq 2\lambda \mathcal{K}\Vert \vartheta_\lambda \Vert_{L^\infty(\overline{\Omega})} 
\end{aligned}
\end{equation*}
where we use \eqref{eq:d/dt-u-lambda}, integration by parts and Lemma \ref{lem:est_alpha(t)} below. Since $\{\vartheta_\lambda\}_{\lambda\in(0,1]}$ are uniformly bounded, by the {\it Dominated Convergence Theorem} $\int_{T\ol\Om}\left(L(y,v,0)+c(H)\right) d\mu_x^0=0$ which completes the proof.
\end{proof}

\begin{lem}\label{lem:est_alpha(t)} Assume \ref{itm:C0''}, \ref{itm:C1}, \ref{itm:C2}--\ref{itm:C6}, there exist positive constants $\mathcal{K}_1, \mathcal{K}_2$ such that
\begin{equation*}
   \mathcal{K}_1 |t| \leq \alpha(t) \leq \mathcal{K}_2 |t|, \qquad t\in (-\infty,0]
\end{equation*}
where the function $\alpha_x^\lambda(t)$ is defined in \eqref{eq:def_alpha(t)}. As a consequence
\begin{equation}\label{eq:est_exp}
    \frac{\mathrm{d}}{\mathrm{d}t}\left(e^{-\lambda \alpha^\lambda_x(t)}\right) \geq 0 \qquad \text{and}\qquad \frac{1}{\lambda \mathcal{K}_2} \leq \int_{-\infty}^0 e^{-\lambda  \alpha^\lambda_x(t)}dt \leq \frac{1}{\lambda \mathcal{K}_1} \qquad t\in (-\infty, 0].
\end{equation}
\end{lem}
\begin{proof} From \ref{itm:C1}, $\partial_u L(y,v,0)$ is actually negative for $(y,v)\in T^*_x\mathcal{U}$. From \ref{itm:C6}, $(y,v)\mapsto \partial_u L(y,v,0)$ is continuous on any compact set of $T\cU$. Since $\gamma^\lambda_x(\cdot)$ is bounded and has a uniform Lipschitz bound (Proposition \ref{prop:opt-formula}), we deduce that
\begin{equation*}
    -\mathcal{K}_2 \leq \partial_u L\left(\gamma^\lambda_x(t),\dot{\gamma}^\lambda_x(t),0\right) \leq -\mathcal{K}_1 ,  \qquad t\in (-\infty, 0].   
\end{equation*}
By the definition of $\alpha^\lambda_x(t)$ from \eqref{eq:def_alpha(t)}, we have $\mathcal{K}_1 |t| \leq \alpha(t) \leq \mathcal{K}_2 |t|$ for $t\leq 0$, and also
\begin{equation}\label{eq:integ}
\begin{aligned}
    \frac{\mbox{d}}{\mbox{d}t}\left(e^{-\lambda \alpha^\lambda_x(t)}\right) 
    &= e^{-\lambda \alpha^\lambda_x(t)}(-\lambda) \left(\frac{\mathrm{d}}{\mathrm{d}t}\alpha_x^\lambda(t)\right)\\
    &=e^{-\lambda \alpha^\lambda_x(t)}\Big( (-\lambda) \partial_u L\big(\gamma^\lambda_x(t), \dot{\gamma}^\lambda_x(t),0\big) \Big) \in \left[\lambda k e^{-\lambda \alpha^\lambda_x(t)}, \lambda\mathcal{K}e^{-\lambda \alpha^\lambda_x(t)}\right].
\end{aligned}
\end{equation}
As $\alpha^\lambda_x(t)\to +\infty$ as $t\to -\infty$, we have $\int_{-\infty}^0 \frac{\mbox{d}}{\mbox{d}t}\left(e^{-\lambda \alpha^\lambda_x(t)}\right) dt = 1$. Using that in \eqref{eq:integ} after integrating over $t\in (-\infty, 0]$ we obtain \eqref{eq:est_exp}.
\end{proof}

The following technical Lemma concerns when we can approximate a viscosity subsolution of $H(x, Dw(x), 0) \leq c$ in $\Omega$ by a smooth approximated subsolution.

\begin{lem}\label{lem:smooth_smaller_domain} Let $H(x,p,0)$ be a Hamiltonian satisfying \ref{itm:C2}, \ref{itm:C4}, and $w\in \mathrm{C}(\overline{\Omega})$ be a Lipschitz viscosity solution to $H(x,Dw(x),0) \leq c$ in $\Omega$. For $\delta>0$ small enough, we define the smaller domain $\Omega^\delta = \{x\in \Omega:\mathrm{dist}(x,\partial\Omega) > \delta\}$.
\begin{itemize}

    \item[$(\mathrm{i})$] If $\Omega$ satisfies \ref{itm:C0''} then $\left(1-\theta^{-1}\delta\right)\Omega \subset \Omega^\delta$.
    
    \item[$(\mathrm{ii})$] If $\Omega$ satisfies either \ref{itm:C0''}, or $\partial \Omega$ is of class $\mathrm{C}^1$ and \ref{itm:C5} holds, then for every $\delta>0$ there exists $w_\delta\in \mathrm{C}^1(\overline{\Omega})$ such that $w_\delta\to w$ a.e. as $\delta\to 0^+$ and $H\left(x,Dw_\delta(x),0\right) \leq c + \mathcal{O}(\delta)$ in $\Omega$.
   
\end{itemize}
\end{lem}
The proof of this Lemma is provided in Appendix. 

\section{Convergence of viscosity solutions in variable domains}\label{sec:3-conv-variable-domains-general} 
We always assume \ref{itm:L} in this section. Let $u_\lambda\in \mathrm{C}(\overline{\Omega}_\lambda)$ be a solution to \eqref{eq:hj-lb}, i.e.,
\begin{equation}\label{eq:C-lambda}
    \begin{cases}
        H(x,Du_\lambda(x),\lambda u_\lambda(x)) \leq C_\lambda, \qquad x\in \Omega_\lambda,\\
        H(x,Du_\lambda(x),\lambda u_\lambda(x)) \geq C_\lambda, \qquad x\in \overline{\Omega}_\lambda,
    \end{cases}
\end{equation}
where $C_\lambda\to c(H)$ as $\lambda \to 0^+$. To get estimate for $u_\lambda$ (by Theorem \ref{thm:perron}), it is necessary to consider the the ergodic problem on $\Omega_\lambda$, that is 
\begin{equation}\label{eq:ergodic}
    \begin{cases}
        \begin{aligned}
            H(x,D  u(x),0)\leq c(\lb),\quad x\in \Om_\lb,\\
            H(x,D  u(x),0)\geq c(\lb),\quad x\in \ol\Om_\lb.\\
        \end{aligned}         
    \end{cases}
\end{equation}
Here $c(\lambda)$ is the ergodic constant, defined as
\begin{equation*}
    c(\lambda) = \inf \left\lbrace c\in \R: H(x,Du(x),0) \leq c\;\text{has a solution in}\;\Omega_\lambda \right\rbrace .
\end{equation*}
We have $c(H) = \lim_{\lambda\rightarrow 0^+}c(\lambda) = \lim_{\lambda\rightarrow 0^+} C_\lambda$. We state the following properties on $c(\lambda)$. 

\begin{prop}\label{prop:aux-conv} Assume \ref{itm:C1'}--\ref{itm:C5}.
\begin{itemize}
\item[$\mathrm{(i)}$] The only value of $e\in\R$ such that 
\beq\label{eq:hj-aux-2}
\left\{
\begin{aligned}
    H(x,D  u(x),e)\leq c(\lb),\quad x\in \Om_\lb,\\
    H(x,D  u(x),e)\geq c(\lb),\quad x\in \ol\Om_\lb.\\
\end{aligned}
\right.
\eeq
possesses viscosity solutions is zero;
\item[$\mathrm{(ii)}$] There exists a uniform constant $K$ depends on $\mathcal{U}$ such that for all $\lambda>0$ then
\begin{equation*}
    \left|\frac{c(\lb)-c(H)}{r(\lb)}\right|\leq K.
\end{equation*}
\end{itemize}
\end{prop}
\begin{proof} For the first part, due to the definition of     Ma\~n\'e's critical value $c(H)$ in \eqref{eq:mane-cri} and \ref{itm:C1'}, there should be a unique $c_\lb(e)$ associated with \eqref{eq:hj-aux-2}. Moreover, $c_\lb(e)$ is continuous and strictly increasing in $e$, which urges that $c_\lb(0)=c(\lb)$. The second part is a direct citation of item (ii) of \cite[Theorem 2.22]{Tu}. 
\end{proof}

If \eqref{eq:assump_ratio} holds then from (ii) of Proposition \ref{prop:aux-conv} we have $|C_\lambda-c(\lambda)| \leq C\lambda$ for $\lambda>0$. Together with Perron's method we can deduce that the family of solutions $\{u_\lambda\}_{\lambda > 0}$ is uniformly bounded and equi-Lipschitz. 

\begin{prop}[Boundedness]\label{prop:bounded} Assume \ref{itm:C1'}--\ref{itm:C5} so that the comparison principle holds. If $|C_\lambda - c(H)|\leq C\lambda$ as $\lambda\to 0^+$ then the solution $u_\lambda$ of \eqref{eq:C-lambda} satisfies that $\Vert u_\lambda \Vert_{L^\infty(\Omega_\lambda)}$ and $\Vert Du_\lambda \Vert_{L^\infty(\Omega_\lambda)}$ are uniformly bounded independent of $\lambda\to 0^+$. 
\end{prop}
\begin{proof} From (ii) of Proposition \ref{prop:aux-conv} we have $|C_\lambda - c(\lambda)| = \mathcal{O}(\lambda)$ as $\lambda \to 0^+$. Let $u\in \mathrm{C}(\overline{\Omega}_\lambda)$ be a solution to \eqref{eq:ergodic} and we define 
\begin{equation*}
    w^\pm(x) = u(x) \pm \Vert u\Vert_{L^\infty(\Omega)} \pm \frac{|C_\lambda - c(\lambda)|}{\lambda\kappa}, \qquad x\in \overline{\Omega}_\lambda.
\end{equation*}
Then $Dw^\pm = Du$ and $\lambda w^- \leq 0 \leq \lambda w^+$, by \ref{itm:C1'} and \eqref{eq:ergodic} we have
\begin{equation*}
    \begin{aligned}
        H(x,Dw^-,\lambda w^-) &+ |C_\lambda - c(\lambda)|
        \leq H\left((x,Du,\lambda w^- + \frac{|C_\lambda - c(\lambda)|}{\kappa}\right) \\
        &= H\left(x,Du, \lambda u - \lambda\Vert u\Vert_{L^\infty(\Omega)} \right) \leq H(x, Du,0) \leq c(\lambda) \qquad\text{in}\;\Omega_\lambda,\\
        H(x,Dw^+,\lambda w^+) &- |C_\lambda - c(\lambda)|
        \geq H\left((x,Du,\lambda w^+ - \frac{|C_\lambda - c(\lambda)|}{\kappa}\right) \\
        &= H\left(x,Du, \lambda u + \lambda\Vert u\Vert_{L^\infty(\Omega)} \right) \geq H(x, Du,0) \geq c(\lambda)
        \qquad\text{on}\;\overline{\Omega}_\lambda.
    \end{aligned}
\end{equation*}
Therefore
\begin{equation*}
    \begin{cases}
        H(x,Dw^-,\lambda w^-)  \leq c(\lambda) - |C_\lambda - c(\lambda)| \leq C_\lambda\qquad \text{in}\;\Omega_\lambda\\
        H(x,Dw^+,\lambda w^+)  \geq c(\lambda) + |C_\lambda - c(\lambda)| \geq C_\lambda \qquad \text{on}\;\overline{\Omega}_\lambda.
    \end{cases}
\end{equation*}
By comparison principle, we have $w^-\leq u_\lambda\leq w^+$ on $\overline{\Omega}_\lambda$, i.e.,
\begin{equation*}
    \Vert u_\lambda\Vert_{L^\infty(\Omega_\lambda)} \leq 2\vert u\Vert_{L^\infty(\Omega_\lambda)} + C\kappa^{-1}
\end{equation*}
for all $\lambda$ small such that $|C_\lambda - c(\lambda)|\leq C\lambda$. Now we observe that from \ref{itm:C3}, any $u$ solve \eqref{eq:ergodic} satisfies the gradient bound $ \Vert Du\Vert_{L^\infty(\Omega_\lambda)} \leq C_\mathcal{U}$ for some constant $C_\mathcal{U}$ depends on $\mathcal{U}$ and $H$. Thus we can choose a solution $u$ of \eqref{eq:ergodic} such that $\min_{\Omega_\lambda} u = 0$ and consequently $\Vert u\Vert_{L^\infty(\Omega_\lambda)} \leq C_\mathcal{U}\mathrm{diam}(U)$, thus $\Vert u_\lambda\Vert_{L^\infty(\Omega_\lambda)}$ is uniformly bounded.
\end{proof}

\begin{proof}[Proof of Theorem \ref{thm:ge}] Assume \eqref{eq:assump_ratio} holds, suppose $u_\lambda$ is the viscosity solution of \eqref{eq:hj-lb} then
\begin{equation*}
    \wt u_\lambda(x):=\frac{u_\lambda((1+r(\lambda))x)}{1+r(\lambda)},\quad x\in\overline{\Omega}
\end{equation*}
should be a viscosity solution of the following
\begin{equation}\label{eq:hj-rect-new}
    \begin{cases}
        \begin{aligned}
        H((1+r(\lambda))x,D  u(x),\lambda(1+r(\lambda))u(x))\leq C_\lambda,\quad x\in \Omega,\\
        H((1+r(\lambda))x,D  u(x),\lambda (1+r(\lambda))u(x))\geq C_\lambda,\quad x\in \overline{\Omega}.\\
        \end{aligned}
    \end{cases}
\end{equation}
Due to Proposition \ref{prop:bounded} and \eqref{eq:assump_ratio}, both $\{\wt u_\lb\}_{\lb>0}$ and $\{u_\lb\}_{\lb>0}$ are uniformly Lipschitz and bounded, then by the {\it Arzel\`a-Ascoli Theorem} any convergent subsequence of $\{\wt u_\lb\}_{\lb>0}$  as $\lambda\rightarrow 0^+$ has to be a viscosity solution of \eqref{eq:hj-0}. Suppose $u_0$  is such a limit of a subsequence, i.e. $\lim_{i\rightarrow +\infty}\wt u_{\lambda_i}= u_0$ locally uniformly in $\mathrm{C}(\overline{\Omega})$, then for any $\mu\in\mathcal{M}$, we have 
\begin{equation*}
\begin{cases}
    \begin{aligned}
    &0= \int_{T\overline{\Omega}}\Big( L(x,v,0)+c(H) \Big)\;d\mu(x,v),\\
    &0\leq \int_{T\overline{\Omega}} \Big(L\left((1+r(\lambda))x,v,\lambda(1+r(\lambda))\wt u_{\lambda}(x)\right) +C_\lambda\Big) \;d\mu(x,v).
    \end{aligned}
\end{cases}
\end{equation*}
Therefore 
\begin{equation}\label{eq:crucial-1}
\begin{aligned}
    0 
    &\leq \int_{T\overline{\Omega}}    \Big(L\left((1+r(\lambda))x,v,\lambda(1+r(\lambda))\wt u_{\lambda}(x)\right) - L(x,v,0)\Big) \;d\mu(x,v)\\
    &=\int_{T\overline{\Omega}}    \Big(L\left((1+r(\lambda))x,v,\lambda(1+r(\lambda))\wt u_{\lambda}(x)\right) - L\left((1+r(\lambda))x,v,0\right)\Big) \;d\mu(x,v) \\
    & 
    +\int_{T\overline{\Omega}}    \Big(L\left((1+r(\lambda))x,v,0\right) - L\left(x,v,0\right)\Big) \;d\mu(x,v) + \Big(C_\lambda - c(H)\Big).
\end{aligned}
\end{equation}
Thanks to \ref{itm:C6} and \ref{itm:C7}, we can divide both side by $\lambda$ and let $\lambda\to 0^+$ along $\{\lambda_i\}$ to obtain
\begin{equation}\label{eq:property-forward}
\begin{aligned}
    \int_{T\ol\Om}u_0(x)\partial_u L(x,v,0)d\mu+\eta\int_{T\ol\Om}\langle\partial_xL(x,v,0),x\rangle  d\mu + \zeta\geq 0
\end{aligned}
\end{equation}
where $\mu\in\cM$ is arbitrarily chosen, thus $u_0 \in \mathcal{E}^{\eta,\zeta}$.\smallskip

To show that $u_0$ is the maximal element of $\mathcal{E}^{\eta,\zeta}$, take $\omega \in \cE^{\eta,\zeta}$ and $x\in \Omega$, we show that $u_0(x) \geq \omega(x)$. We consider $\lambda$ small so that $x\in \Omega_\lambda$. Due to Proposition \ref{prop:opt-formula}, for any $x\in \ol\Om_\lb$, there exists a Lipschitz curve $\gamma_x^\lb:(-\infty, 0]\rightarrow \ol\Om_\lb$ ending with it, such that 
\begin{align*}
    \frac{\mathrm{d}}{\mathrm{d}t} u_\lb(\gamma_x^\lb(t))
    &=L\left(\gamma_x^\lb(t),\dot\gamma_x^\lb(t),\lb  u_\lb(\gamma_x^\lb(t))\right)+C_\lambda\\
    &= L\left(\gamma_{x}^\lb(t),\dot\gamma_{x}^\lb(t),0\right)+ C_\lambda +\lb  u_\lb\left(\gamma_{x}^\lb(t)\right) \partial_u L\left(\gamma_{x}^\lb(t),\dot\gamma_{x}^\lb(t),0\right)+\Theta_x^\lb(t)
\end{align*}
for a.e. $t\in(-\infty,0]$ with
\begin{align*}
        \Theta_x^\lb(t):=L\left(\gamma_x^\lb(t),\dot\gamma_x^\lb(t),\lb  u_\lb(\gamma_x^\lb(t))\right)
        &-L\left(\gamma_{x}^\lb(t),\dot\gamma_{x}^\lb(t),0\right)\\
        &-\lb  u_\lb\left(\gamma_{x}^\lb(t)\right) \partial_u L\left(\gamma_{x}^\lb(t),\dot\gamma_{x}^\lb(t),0\right).
\end{align*}
Similar to the proof of Theorem \ref{thm:fix}, we define 
\begin{equation}\label{eq:alpha-2-new}
    \alpha_{x}^\lb(t):=\int_0^t\partial_u L\left(\gamma_{x}^\lb(s),\dot\gamma_{x}^\lb(s),0\right)ds\qquad\text{then}\qquad \int_{-\infty}^0 \Theta_x^\lb(t)e^{-\lb\alpha_x^\lb(t)} dt\rightarrow 0
\end{equation}
as $\lb\rightarrow 0^+$ due to \ref{itm:C6}. Let us define
\begin{equation*}
        \wt\om_\lb(y):=(1+r(\lb))\om\left(\frac{y}{1+r(\lambda)}\right), \qquad y\in \overline{\Omega}_\lambda
\end{equation*}
then it is a subsolution of 
\begin{equation*}
    H\left(\frac{y}{1+r(\lambda)}, D\wt \omega_\lambda(y),0\right) \leq c(H), \quad y\in \Omega_\lambda.
\end{equation*}
By Lemma \ref{lem:smooth_smaller_domain}, for each $\delta > 0$ there is $\wt \omega^\delta_\lambda \in \mathrm{C}^1(\overline{\Omega}_\lambda)$ such that $\Vert \wt \omega^\delta_\lambda - \wt \omega_\lambda\Vert_{L^\infty(\Omega_\lambda)} < \delta$ and 
\begin{equation*}
    H\left(\frac{x}{1+r(\lambda)}, D\wt \omega^\delta_\lambda(x),0\right) \leq c(H) + \mathcal{O}(\delta), \quad x\in \Omega_\lambda .
\end{equation*}
Therefore
\begin{equation*}
    \frac{\mathrm{d}}{\mathrm{d}t}  \left(\wt \omega^\delta_\lambda(\gamma^\lambda_x(t))\right) \leq L\left(\frac{\gamma^\lambda_x(t)}{1+r(\lambda)},\dot{\gamma}^\lambda_x(t),0\right) + c(H) + \mathcal{O}(\delta)
\end{equation*}
and thus due to the same procedure as in the proof of Theorem \ref{thm:fix}, we get
\begin{equation*}
    \begin{aligned}
        \frac{\mbox{d}}{\mbox{d}t}\left(\wt\omega^\delta_\lb(\gamma^\lb_{x}(t))\right)
        &\leq
        \frac{\mbox{d}}{\mbox{d}t}\left( u_{\lb}(\gamma^\lb_{x}(t))\right)-\lb  u_\lb(\gamma_{x}^\lb(t))\partial_u L(\gamma_{x}^\lb(t),\dot\gamma_{x}^\lb(t),0)-\Theta_x^\lb(t)\\
        & +L\left(\frac{\gamma_{x}^\lb(t)}{1+r(\lb)},\dot\gamma_{x}^\lb(t),0\right)-L\left(\gamma_{x}^\lb(t),\dot\gamma_{x}^\lb(t),0\right) - \mathcal{O}(\delta)\\
        & -\big(C_\lambda - c(H)\big).
    \end{aligned}
\end{equation*}
Therefore
\begin{equation*}
    \begin{aligned}
        \frac{\mbox{d}}{\mbox{d}t}\left(e^{-\lambda\alpha^\lambda_x(t)}\wt\omega^\delta_\lb(\gamma^\lb_{x}(t))\right) 
             &\leq \frac{\mbox{d}}{\mbox{d}t}\left(e^{-\lambda\alpha^\lambda_x(t)} u_\lb(\gamma^\lb_{x}(t))\right) - e^{-\lambda\alpha^\lambda_x(t)}\Theta^\lambda_x(\gamma^\lambda_x(t)) - e^{\lambda\alpha^\lambda_x(t)}\mathcal{O}(\delta)\\
             &\quad - e^{-\lambda\alpha^\lambda_x(t)}\lambda \wt\omega^\delta_\lambda(\gamma^\lambda_x(t))\partial_uL(\gamma^\lambda_x(t),\dot{\gamma}^\lambda_x(t),0)\\
             &\quad + e^{-\lambda\alpha^\lambda_x(t)}\left[L\left(\frac{\gamma_{x}^\lb(t)}{1+r(\lb)},\dot\gamma_{x}^\lb(t),0\right)-L\left(\gamma_{x}^\lb(t),\dot\gamma_{x}^\lb(t),0\right)\right]\\
             &\quad - e^{-\lambda \alpha^\lambda_x(t)}\big(C_\lambda - c(H)\big).
    \end{aligned}
\end{equation*}
Taking integration over $t\in (-\infty, 0]$, we deduce that 
\begin{equation*}
    \begin{aligned}
         u_\lb(x)\geq \wt\om^\delta_\lb(x) 
        &+\int_{-\infty}^0 \Theta_x^\lb(t)e^{-\lb\alpha_x^\lb(t)} dt + \Big(C_\lambda - c(H) + \mathcal{O}(\delta)\Big)\int_{-\infty}^0 e^{-\lb\alpha_x^\lb(t)} dt\\
                        & +\lb\int_{-\infty}^0 e^{-\lb\alpha_x^\lb(t)}dt\cdot \int_{T\ol\Om_\lb}\wt\om^\delta_\lb(y)\partial_uL(y,v,0) d\mu_x^\lb(y,v)\\
                        & +\int_{-\infty}^0 \left[L\left(\gamma_{x}^\lb(t),\dot\gamma_{x}^\lb(t),0\right)-L\left(\frac{\gamma_{x}^\lb(t)}{1+r(\lb)},\dot\gamma_{x}^\lb(t),0\right)\right]e^{-\lb\alpha_x^\lb(t)} dt
    \end{aligned}
\end{equation*}
where $\mu_x^\lb\in\cP(T\ol\Om_\lb,\R)$ is defined by 
\begin{equation*}
    \int_{T\ol\Om_\lb} f(y,v)d\mu_x^\lb(y,v):=
    \frac
    {\displaystyle
        \int_{-\infty}^0f(\gamma_x^\lb(t),\dot\gamma_x^\lb(t))e^{-\lb\alpha_x^\lb(t)}dt
    }
    {\displaystyle
        \int_{-\infty}^0e^{-\lb\alpha_x^\lb(t)}dt
    },
    \qquad \text{for}\; f\in \mathrm{C}_c(\R^{2n})
\end{equation*}
and $\alpha^\lambda_x(\cdot)$ is defined as in \eqref{eq:alpha-2-new}. We note that 
\begin{equation*}
    \int_{T\ol\Om_\lb}\partial_u L(y,v,0)d\mu_x^\lb(y,v)=
    \frac{\displaystyle
        1
    }
    {\displaystyle
        -\lb \int_{-\infty}^0e^{-\lb\alpha_x^\lb(t)}dt
    },
\end{equation*}
therefore
\begin{equation*}
    \lb\int_{-\infty}^0 e^{-\lb\alpha_x^\lb(t)}dt\cdot \int_{T\ol\Om_\lb}\wt\om^\delta_\lb(y)\partial_uL(y,v,0) d\mu_x^\lb(y,v) = 
    -\frac
    {\displaystyle
        \int_{T\overline{\Omega}_\lambda}\wt \omega^\delta_\lambda(y)\partial_u L(y,v,0)d\mu^\lambda_x(y,v)
    }
    {\displaystyle
        \int_{T\overline{\Omega}_\lambda}\partial_u L(y,v,0)d\mu^\lambda_x(y,v).
    }
\end{equation*}
Also, we have 
\begin{equation*}
    \begin{aligned}
        &\int_{-\infty}^0 \left[L\left(\gamma_{x}^\lb(t),\dot\gamma_{x}^\lb(t),0\right)-L\left(\frac{\gamma_{x}^\lb(t)}{1+r(\lb)},\dot\gamma_{x}^\lb(t),0\right)\right]e^{-\lb\alpha_x^\lb(t)} dt\\
        &\qquad\qquad\qquad = \int_{-\infty}^0 e^{-\lambda \alpha^\lambda_x(t)}dt\cdot\int_{T\overline{\Omega}_\lambda} \left(L(y,v,0) - L\left(\frac{y}{1+r(\lambda)},v,0\right)\right)d\mu^\lambda_x(y,v)\\
        &\qquad\qquad\qquad = 
        \frac
        {\displaystyle
            \int_{T\overline{\Omega}_\lambda} \left(L(y,v,0) - L\left(\frac{y}{1+r(\lambda)},v,0\right)\right)d\mu^\lambda_x(y,v)
        }
        {\displaystyle
            -\lambda \int_{T\overline{\Omega}_\lambda} \partial_uL(y,v,0)d\mu^\lambda_x(y,v)
        }
    \end{aligned}
\end{equation*}
and
\begin{equation*}
    \Big(C_\lambda - c(H) \Big)\int_{-\infty}^0 e^{-\lb\alpha_x^\lb(t)} dt = -\frac
    {\displaystyle
        C_\lambda - c(H)
    }
    {\displaystyle
        -\lambda \int_{T\overline{\Omega}_\lambda} \partial_uL (y,v,0)d\mu^\lambda_x(y,v)
    }.
\end{equation*}
Besides, we have $\wt\omega^\delta_\lambda \to \wt\omega_\lambda$ as $\delta \to 0^+$ uniformly, and 
\begin{equation*}
    \begin{aligned}
    \int_{T\ol\Om_\lb}\wt\om_\lb(y)\partial_uL(y,v,0) d\mu_x^\lb(y,v) &= (1+r(\lb))\int_{T\ol\Om}\om(y)\partial_uL((1+r(\lb))y,v,0) d\wt\mu_x^\lb(y,v),\\
    \int_{T\ol\Om_\lb}\partial_uL(y,v,0) d\mu_x^\lb(y,v) &= \int_{T\ol\Om}\partial_uL((1+r(\lb))y,v,0) d\wt\mu_x^\lb(y,v).
\end{aligned}
\end{equation*}
where we take the rectified measure $\wt\mu_x^\lb\in\cP(T\ol\Om,\R)$ by 
\begin{equation*}
    \int_{T\ol\Om} f(y,v)d\wt\mu_x^\lb(y,v):= \int_{T\ol\Om_\lb} f\Big(\frac y{1+r(\lb)},v\Big)d\mu_x^\lb(y,v), \qquad\text{for}\; f\in \mathrm{C}_c(\R^{2n}).     
\end{equation*}
Consequently, 
\begin{align*}
     u_\lambda(x) 
    \geq  \wt\omega^\delta_\lambda(x)
    &+ \int_{-\infty}^0 \Theta^\lambda_x(t) e^{\lambda \alpha^\lambda_x(t)}dt +\Big(C_\lambda - c(H) + \mathcal{O}(\delta)\Big)\int_{-\infty}^0 e^{-\lb\alpha_x^\lb(t)} dt\\
    &-\frac
    {\displaystyle
        \int_{T\overline{\Omega}_\lambda}\wt \omega^\delta_\lambda(y)\partial_u L\big(y,v,0\big)d\mu^\lambda_x(y,v)
    }
    {\displaystyle
        \int_{T\overline{\Omega}_\lambda}\partial_u L\big(y,v,0\big)d\mu^\lambda_x(y,v).
    } \\
    &
    \frac
        {\displaystyle
            \int_{T\overline{\Omega}_\lambda} \left(L(y,v,0) - L\left(\frac{y}{1+r(\lambda)},v,0\right)\right)d\mu^\lambda_x(y,v)
        }
        {\displaystyle
            -\lambda \int_{T\overline{\Omega}_\lambda} \partial_uL(y,v,0)d\mu^\lambda_x(y,v)
        }.
\end{align*}
Let $\delta\to 0^+$ and using the rectified measure $\wt \mu^\lambda_x$ instead of $\mu^\lambda_x$ we obtain
\begin{equation}\label{eq:crucial}
    \begin{aligned}
     u_\lambda(x) 
    \geq  \wt\omega_\lambda(x)
    &+ \int_{-\infty}^0 \Theta^\lambda_x(t) e^{\lambda \alpha^\lambda_x(t)}dt -\frac
    {\displaystyle
        \frac{C_\lambda - c(H)}{\lambda}
    }
    {\displaystyle
        \int_{T\overline{\Omega}_\lambda} \partial_uL ((1+r(\lambda))y,v,0)d\wt\mu^\lambda_x(y,v)
    }\\
    &-\frac
    {\displaystyle
        (1+r(\lambda))
        \int_{T\overline{\Omega}}\wt \omega_\lambda(y)\partial_u L\big((1+r(\lambda))y,v,0\big)d\wt\mu^\lambda_x(y,v)
    }
    {\displaystyle
        \int_{T\overline{\Omega}}\partial_u L\big((1+r(\lambda))y,v,0\big)d\wt\mu^\lambda_x(y,v).
    } \\
    &- \left(\frac{r(\lambda)}{\lambda}\right)
    \frac
    {\displaystyle
        \int_{T\overline{\Omega}} \left( \frac{L\big((1+r(\lambda)y,v,0)\big) - L(y,v,0))}{r(\lambda)}\right)d\wt\mu^\lambda_x(y,v)
    }
    {\displaystyle
        \int_{T\overline{\Omega}} \partial_uL\big((1+r(\lambda)y,v,0)\big)d\wt\mu^\lambda_x(y,v)
    }.
    \end{aligned}
\end{equation}
By the 
Lemma \ref{lem:wt_mu-lambda-x-new} below, any accumulating measure 
$\mu_x^0$ of $\{\wt\mu_x^{\lb}\}_{\lb\rightarrow 0^+}$ (in the sense of weak topology) has to be contained in $\cM$. Therefore, if $\om\in\cE^{\eta,\zeta}$,  taking $\lb\rightarrow 0^+$ in \eqref{eq:crucial} yields
\begin{equation}\label{eq:eta-new}
    \begin{aligned}
        u_0(x) 
        &\geq \omega(x) 
        - \frac
            {\displaystyle
                \int_{T\overline{\Omega}} \omega(y)\partial_u L(y,v,0)d \mu^0_x(y,v)
                + \eta \int_{T\overline{\Omega}} \left\langle\partial_x L(y,v,0),y\right\rangle d\mu^0_x(y,v) + \zeta
            }
            {\displaystyle
                \int_{T\overline{\Omega}} \partial_u L(y,v,0)d \mu^0_x(y,v)
            }.
    \end{aligned}
\end{equation}
From \ref{itm:C1'} we have $\partial_uL(y,v,0) < 0$ and
\begin{equation*}
    \int_{T\overline{\Omega}} \omega(y)\partial_u L(y,v,0)d\mu^0_x(y,v)
                + \eta \int_{T\overline{\Omega}} \left\langle\partial_x L(y,v,0),y\right\rangle d \mu^0_x(y,v) + \zeta \geq 0
\end{equation*}
due to $\omega \in \mathcal{E}^{\eta,\zeta}$, we deduce from \eqref{eq:eta-new} that $ u_0(x) \geq \omega(x)$. Since $x\in\overline{\Omega}$ is freely chosen, $u_0 = \max \mathcal{E}^{\eta,\zeta}$. \medskip

Conversely, assume that there exists $u_0\in \mathrm{C}(\overline{\Omega})$ such that $u_\lambda \to u_0$ then by stability $u_0$ is a solution to \eqref{eq:hj-0} and also $\wt u_\lambda \to u_0$ as well. We proceed as in the previous argument, but without assumption \eqref{eq:assump_ratio}. Firstly, in \eqref{eq:crucial-1} we have the following instead
\begin{equation}\label{eq:property-forward-liminf}
\begin{aligned}
    \int_{T\overline{\Omega}} u_0(y)\partial_u L(y,v,0)d\mu+\eta\int_{T\overline{\Omega}}\langle\partial_xL(y,v,0),y\rangle  d\mu + \liminf_{\lambda \to 0^+} \left(\frac{C_\lambda - c(H)}{\lambda}\right)\geq 0
\end{aligned}
\end{equation}
for all $\mu \in \mathcal{M}$. Secondly, proceed as in \eqref{eq:crucial} but we choose $\omega = u_0$ to start with, then
\begin{equation}\label{eq:property-forward-limsup}
    \int_{T\overline{\Omega}} u_0(y)\partial_u L(y,v,0)d\mu^0+\eta\int_{T\overline{\Omega}}\langle\partial_xL(y,v,0),y\rangle  d\mu^0 + \limsup_{\lambda \to 0^+} \left(\frac{C_\lambda - c(H)}{\lambda}\right)\leq 0
\end{equation}
for some $\mu^0 \in \mathcal{M}$, where we use the fact that $\partial_u L(y,v,0) < 0$ again here. From \eqref{eq:property-forward-liminf} and \eqref{eq:property-forward-limsup} we deduce that 
\begin{equation*}
    \zeta  =\lim_{\lambda \to 0^+} \left(\frac{C_\lambda - c(H)}{\lambda}\right) = - \int_{T\overline{\Omega}} \Big(u_0(y)\partial_uL(y,v,0) + \eta \big\langle \partial_x L(y,v,0), y\big\rangle\Big)d \mu^0(y,v).
\end{equation*}
The proof is therefore complete.
\end{proof}

\begin{lem}\label{lem:wt_mu-lambda-x-new} In the context of Theorem \ref{thm:ge}, if $x\in \Omega$ then any accumulating measure $\mu_x^0$ of $\{\wt\mu_x^{\lambda}\}_{\lambda\rightarrow 0^+}$ in the sense of weak* topology has to be contained in $\mathcal{M}$.
\end{lem}

\begin{proof}[Proof of Lemma \ref{lem:wt_mu-lambda-x-new}] We show that $ \mu_x^0$ is holonomic and its action against $L(x,v,0)$ is $-c(H)$. Let $x\in \Omega$, we can assume $x\in \overline{\Omega}_\lambda$ for $\lambda>0$ small enough. The corresponding minimizing Lipschitz curve $\gamma^\lambda_x(\cdot)$ such that $|\dot{\gamma}^\lambda_x(t)|\leq C$ independent of $\lambda$ as $C$ only depends on $\mathcal{U}$ and universal constants. Denote $\alpha^\lambda_x(t) = \int_0^t \partial_uL(\gamma^\lambda_x(t), \dot{\gamma}^\lambda_x(t), 0)dt$ for $t\leq 0$ as previously, then
\begin{equation*}
\begin{aligned}
    (1+r(\lambda))\int_{T\overline{\Omega}} |v|d\wt\mu^\lambda_x(y,v) 
        &=\int_{T\overline{\Omega}_\lambda} |v|d\mu^\lambda_x(y,v) \\
        &= \left(\int_{-\infty}^0 e^{-\lambda \alpha_x^\lambda(t)}dt\right)^{-1}\left(\int_{-\infty}^0 |\dot{\gamma}^\lambda_x(t)|e^{-\lambda \alpha_x^\lambda(t)}dt\right) \leq C.
\end{aligned}
\end{equation*}
Letting $\lambda\to 0^+$ we deduce that $\int_{T\overline{\Omega}} |v|d\wt\mu^0_x(y,v) \leq C$. For $\phi\in \mathrm{C}^1(\overline{\Omega})$, we define $\phi^\lambda(y) = \phi\big((1+r(\lambda)^{-1}y) \big)$ which belongs to $\mathrm{C}^1(\overline{\Omega}_\lambda)$. By definition we have
\begin{equation*}
    \frac{1}{1+r(\lambda)}\int_{T\overline{\Omega}} \langle D\phi(y), v\rangle\;d\wt\mu^\lambda_x(y,v)  = \int_{T\overline{\Omega}_\lambda} \langle D\phi^\lambda(y), v\rangle\;d\mu^\lambda_x(y,v).
\end{equation*}
Similar to Lemma \ref{lem:mat-mea}, the right hand side vanishes as $\lambda\to 0^+$, therefore passing to the limit we obtain $\int_{T\overline{\Omega}} \langle D\phi(y), v\rangle d\mu^0_x(y,v) = 0$. Finally, we have
\begin{equation*}
    \begin{aligned}
    &\int_{T\overline{\Omega}}\left[L\big((1+r(\lambda))y,v,\lambda u_\lambda((1+r(\lambda))y)\big)+C_\lambda\right]d\wt\mu^\lambda_x = \int_{T\overline{\Omega}_\lambda} \left[L\big(y,v,\lambda u_\lambda(y)\big)+C_\lambda\right]d\mu^\lambda_x.
    \end{aligned}
\end{equation*}
The right hand side vanishes as $\lambda\to 0^+$ due to an argument similar to Lemma \ref{lem:mat-mea}, thus $\int_{T\overline{\Omega}} L(y,v,0)d\mu^0_x = -c(H)$ since $u_\lambda$ is uniformly bounded and $C_\lambda\to c(H)$ as $\lambda \to 0^+$.
\end{proof}

Now we provide a proof of the convergence of $u_\lambda$ using merely comparison principle for the special case where $\eta = 0$ in \ref{itm:L}. In this case $\zeta = 0$ in \eqref{eq:assump_ratio} due to (iv) of Remark \ref{rmk:case-c(lambda)}. 
\begin{proof}[Proof of Corollary \ref{cor:eta=0}] Recall that $\wt u_\lambda \in \mathrm{C}(\overline{\Omega})$ solves \eqref{eq:hj-rect-new}, and we have a priori estimate bound $\Vert \wt u_\lambda\Vert_{L^\infty(\Omega)} + \Vert D\wt u_\lambda\Vert_{L^\infty(\Omega)}\leq C_\mathcal{U}$, thus by \ref{itm:C45'} there exists a constant $C'_\mathcal{U}$ such that 
\begin{equation}\label{eq:mod}
    \begin{cases}
        H\big(x,D\wt u_\lambda, \lambda (1+r(\lambda))\wt u_\lambda\big) \leq C_\lambda + C'_\mathcal{U}|r(\lambda)|\qquad \text{in}\;\Omega,\\
        H\big(x,D\wt u_\lambda, \lambda (1+r(\lambda))\wt u_\lambda\big) \geq C_\lambda - C'_\mathcal{U}|r(\lambda)|\qquad \text{on}\;\overline{\Omega}.
    \end{cases}
\end{equation}
Let 
\begin{equation*}
    v^\pm_\lambda = \wt u_\lambda \pm |r(\lambda)|C_{\mathcal{U}} \pm  \left(\frac{C'_\mathcal{U}}{\kappa}\right)\frac{|r(\lambda)|}{\lambda} \pm \frac{|C_\lambda - c(H)|}{\lambda \kappa} \qquad\text{in}\;\mathrm{C}(\overline{\Omega}). 
\end{equation*}
Using \ref{itm:C1'} and \eqref{eq:mod} we have
\begin{equation*}
    \begin{aligned}
        H\left(x, Dv^-_\lambda, \lambda v^-_\lambda\right) &= H\left(x, D\wt u_\lambda, \lambda \wt u_\lambda - \lambda|r(\lambda)|C_{\mathcal{U}}
        - \frac{C'_{\mathcal{U}}}{\kappa}|r(\lambda)|-\frac{|C_\lambda - c(H)|}{\kappa}\right) \\
        & \leq 
        H\left(x, D\wt u_\lambda, \lambda (1+r(\lambda))\wt u_\lambda \right) - C'_\mathcal{U}|r(\lambda)|- |C_\lambda - c(H)| \leq c(H)\quad\text{in}\;\Omega,\\
        H\left(x, Dv^+_\lambda, \lambda v^+_\lambda\right) &= H\left(x, D\wt u_\lambda, \lambda \wt u_\lambda + \lambda|r(\lambda)|C_{\mathcal{U}}
        + \frac{C'_{\mathcal{U}}}{\kappa}|r(\lambda)|+\frac{|C_\lambda - c(H)|}{\kappa}\right) \\
        & \geq 
        H\left(x, D\wt u_\lambda, \lambda (1+r(\lambda))\wt u_\lambda \right) + C'_\mathcal{U}|r(\lambda)|+ |C_\lambda - c(H)| \geq c(H)\quad\text{on}\;\overline{\Omega}.
    \end{aligned}
\end{equation*}
By comparison principle with solution $\vartheta_\lambda$ 
of \eqref{eq:thm:fix-ge-equation} we obtain $v_\lambda^- \leq \vartheta_\lambda \leq v^+_\lambda$ on $\overline{\Omega}$, i.e.,
\begin{equation*}
    \left\Vert \wt u_\lambda - \vartheta_\lambda\right\Vert_{L^\infty(\Omega)} \leq |r(\lambda)|C_\mathcal{U} + \left(\frac{C'_\mathcal{U}}{\kappa}\right)\frac{|r(\lambda)|}{\lambda} +\frac{|C_\lambda - c(H)|}{\lambda \kappa} 
\end{equation*}
Using $\eta = 0$ in \ref{itm:L} and $\zeta = 0$ in \eqref{eq:assump_ratio} we obtain the conclusion.
\end{proof}

\section{Generalized convergence for arbitrary \texorpdfstring{$C_\lb$}\; value}\label{sec:4-generalized-conv-c}

In this section, we first show how to generalize the conclusions in Theorem \ref{thm:ge} to the setting with $C_\lambda \to c$ as $\lambda \to 0^+$ such that 
$c$ and $c(H)$ may not be the same. We then give a more essential criterion of the limit $\zeta = \lim_{\lambda\to 0^+}(C_\lambda - c)/\lambda$ in \eqref{eq:assump_ratio}.

\subsection{Convergence for generalized $C_\lambda$}
We present the following useful conclusion about the ergodic constant for a slightly different Hamiltonian.

\begin{lem}\label{lem:admis}
For $\Omega$ satisfying \ref{itm:C0''}, Hamiltonian $H$ satisfying \ref{itm:C1'}, \ref{itm:C2}, \ref{itm:C3} and any $c\in\R$, there exists a maximal $h_0(c)\in\R$ such that the equation 
\begin{equation*}
    H\left(x, D u(x), h_0(c)\right)=c,\quad x\in\Om
\end{equation*}
admits continuous subsolutions. Furthermore, $h_0(c(H)) = 0$, $c\mapsto h_0(c)$ is strictly increasing and
\begin{equation}\label{eq:Lipschitz-h0c}
    0 <  \frac{h_0(c_2) - h_0(c_1)}{c_2-c_1} \leq \frac{1}{\kappa}
\end{equation}
for $c_1\neq c_2$ in $\mathbb{R}$, where $\kappa$ is the constant from \ref{itm:C1'}.
\end{lem}
\begin{proof} Due to \ref{itm:C1'}, \ref{itm:C2}, \ref{itm:C3} and the compactness of $\ol\Om\subset\cU$, there always exists a couple $(h_-, h_+)\in \R^2$ such that $h_-\leq h_+$, and
\begin{equation*}
    H(x,0, h_-)\leq c,\quad\forall x\in\Om, \qquad\text{and}\qquad \min_{(x, p)\in T^*\Om}H(x,p,h_+)> c.
\end{equation*}
Consequently, $h_0(c)\in[h_-,h_+]$ is finite. On the other side, due to the definition in \eqref{eq:mane-cri}, $c$ 
is actually the Ma\~n\'e's critical value of 
the Hamiltonian $H(\cdot,\cdot, h_0(c)): T^*\Om\rightarrow \R$, 
which implies the uniqueness of $h_0(c)$ due to item (i) 
of our Proposition \ref{prop:aux-conv}. As $c(H)$ is the Ma\~n\'e's critical value of $H(\cdot, \cdot, 0)$ we obtain $h_0(c(H)) = 0$ by the uniqueness following from (i) 
of Proposition \ref{prop:aux-conv}.

Let $c_1 < c_2$ be any two real numbers, by the definition of Ma\~n\'e's critical value and \ref{itm:C1'}, it is clear that $h_0(c_1) < h_0(c_2)$. Let $u\in \mathrm{C}(\Omega)$ be a solution to $H(x,Du(x), h_0(c_2)) \leq c_2$ in $\Omega$.
From \ref{itm:C1'} we have
\begin{equation*}
    H(x,Du(x), h_0(c_1)) - \kappa h_0(c_1) \leq H(x,Du(x), h_0(c_2)) - \kappa h_0(c_2) \qquad
    \text{in}\;\Omega.
\end{equation*}
Therefore
\begin{equation*}
    H(x,Du(x),h_0(c_1)) \leq c_2 - \kappa \big(h_0(c_2)-h_0(c_1)\big) \qquad \text{in}\;\Omega.
\end{equation*}
By the definition of the Ma\~n\'e's critical value $c_1 \leq c_2 - \kappa \left(h_0(c_2)-h_0(c_1)\right)$ and thus \eqref{eq:Lipschitz-h0c} follows.
\end{proof}

As we can see, if $c(u)$ is the Ma\~n\'e's critical value of the Hamiltonian $H(x,p,u)$ (for any $u\in\R$), then $c\mapsto h_0(c)$ is actually the inverse function of $u\mapsto c(u)$. Consequently, we have the following asymptotic conclusion related with the Mather measures:

\begin{lem}\label{lem:Mc_conv} Let $\Omega$ satisfy \ref{itm:C0''}, Hamiltonian $H$ satisfy \ref{itm:C1'}, \ref{itm:C2}, \ref{itm:C3} and $c$ be in a neighborhood of $c(H)$. Let $\mathcal{M}, \mathcal{M}_c$ be the set of Mather measures associated with $H(\cdot, \cdot, 0)$ and $H(\cdot, \cdot, h_0(c))$ on $\Omega$, respectively. If $\nu_c\in \mathcal{M}_c$ be a sequence of measures then as $c\to c(H)$, there exists $\nu \in \mathcal{M}$ such that $\nu_c \rightharpoonup \nu$ (along some subsequences) weakly in the sense of measure.
\end{lem}
\begin{proof} Let $\nu$ be a weak limit of of $\nu_c$ along a subsequence, it is clear that $\nu$ is holonomic. By definition of $\nu_c\in \mathcal{M}_c$ we have ${\int_{\Omega}} \left( L\left(x,v,h_0(c)\right) + c\right)\;d\nu_c(x,v) = 0$. By the strong convergence of $L(x,v,h_0(c))\to L(x,v,0)$ as $c\to c(H)$ and $\nu_c\rightarrow \nu$ we obtain the conclusion.
\end{proof}

We present an analogue of Theorem \ref{thm:fix} where $c(H)$ is being replaced by $c$.

\begin{thm}\label{thm:c-ener-fix} Suppose $\Omega$ satisfy \ref{itm:C0''}, $r(\lambda)$ satisfy \ref{itm:L}, $H$ satisfying \ref{itm:C1'},\ref{itm:C2}--\ref{itm:C7} and $c\in\R$, the solution $u_\lambda$ of 
\begin{equation}\label{eq:hj-lb-c-fix} 
    \begin{cases}
        \begin{aligned}
            &H\left(x,Du, \lambda u\right) \leq c \qquad\text{in}\;\Omega,\\
            &H\left(x,Du, \lambda u\right) \geq c \qquad\text{on}\;\overline{\Omega}.
        \end{aligned}
    \end{cases}
\end{equation}
should satisfy $\lambda u_\lambda(\cdot)\rightarrow h_0(c)$ as $\lambda \to 0^+$ for certain constant $h_0(c)\in\R$. Furthermore
\begin{equation*}
    \lim_{\lb\rightarrow 0^+} \left(u_\lb(x)-\frac{h_0(c)}{\lb}\right)=u_c(x)
\end{equation*}
 where 
\begin{equation*}
   u_{c}(x)=\sup\cE_c
\end{equation*}
with
\begin{equation*}
\begin{aligned}
    \mathcal{E}_c &:= \Big\{w\in \mathrm{C}(\Omega,\R):H(x, Dw,h_0(c))\leq c\;\text{in}\;\Omega\;\text{in the viscosity sense}, \\
     &\qquad\qquad\qquad\qquad\;\int_{T\ol\Om}\partial_u L(x,v,h_0(c))w(x)d\mu(x,v)\geq 0,\;\forall \mu\in\cM_c\Big\}
\end{aligned}
\end{equation*}
and $\cM_c$ is the set of Mather measures associated with Hamiltonian $H(\cdot,\cdot, h_0(c)): T^*\Om\rightarrow\R$.
\end{thm}

\begin{proof}[Proof of Theorem \ref{thm:c-ener-fix}]Suppose $u_0(x)$ is a solution of 
\begin{equation}\label{eq:cell-c-Omega}
    \begin{cases}
        \begin{aligned}
            H(x,D  u(x),h_0(c))\leq c,\quad x\in \Om,\\
            H(x,D  u(x),h_0(c))\geq c,\quad x\in \ol\Om,
        \end{aligned}
    \end{cases}
\end{equation}
we conclude that 
\begin{equation*}
    \check u_0:= u_0-\Vert u_0\Vert_{L^\infty(\Omega)}+\frac{h_0(c)}{\lb}
    \;\;\left(\text{resp.}\;\hat u_0:=u_0+\Vert u_0\Vert_{L^\infty(\Omega)} +\frac{h_0(c)}{\lb} \right)
\end{equation*}
is a viscosity subsolution (resp. supersolution) of \eqref{eq:hj-lb-c-fix}.
Consequently, $\check u_0\leq u_\lb\leq \hat u_0$ then further $\lim_{\lb\rightarrow 0^+}\lb u_\lb=h_0(c)$ due to the Comparison Principle. Besides, 
\begin{equation*}
    \left\Vert u_\lb-\frac{h_0(c)}{\lb}\right\Vert_{L^\infty(\Omega)}
    \leq 
    2 \Vert u_0\Vert_{L^\infty(\Omega)}
    <+\infty,\quad \forall \lb>0. 
\end{equation*}
That implies $\left\{u_\lambda-\lambda^{-1} h_0(c)\right\}_{\lambda>0}$ are uniformly bounded. 
If we impose $\lb\in[0,1]$, then 
\begin{equation*}
    \left\Vert \lb u_\lambda - h_0(c)\right\Vert_{L^\infty(\Omega)} \leq C\lambda,\quad\forall \lb\in[0,1]
\end{equation*}
where $C = 2\Vert u_0\Vert_{L^\infty(\Omega)}$. Consequently, for any  $\lb\in[0,1]$ fixed, due to \ref{itm:C1'} we have
\begin{equation*}
    H\left(x,D u_\lambda(x),h_0(c)-C\lambda\right)
    \leq H(x, D u_\lambda(x),\lb u_\lb(x))
    \leq c,\quad a.e.\; x\in\Om.
\end{equation*}
Using \ref{itm:C3} we obtain the boundedness of $\{Du_\lambda\}$, thus $\{u_\lambda - \lambda^{-1}h_0(c)\}_{\lb>0}$ are uniformly Lipschitz. 
Due to the Arzel\`a-Ascoli Theorem, 
the accumulating function of $\wt u_\lambda:=u_\lb- \lambda^{-1}h_0(c)$ as 
$\lb\rightarrow 0^+$ exists, so it suffices to prove its uniqueness. Actually, $\wt u_\lb$ is the viscosity solution of the equation
\begin{equation*}
    \begin{cases}
        \begin{aligned}
            H(x,D  w(x),\lb w(x)+h_0(c))\leq c,\quad x\in \Om,\\
            H(x,D  w(x),\lb w(x)+h_0(c))\geq c,\quad x\in \ol\Om.
        \end{aligned}
    \end{cases}
\end{equation*}
So we can directly apply Theorem \ref{thm:fix} and get $u_{c}(x)=\lim_{\lb\rightarrow 0^+}\wt u_\lb$ which is a viscosity solution of \eqref{eq:cell-c-Omega} and can be expressed as in our assertion. 
\end{proof}

With the help of Theorem \ref{thm:c-ener-fix}, the proof of Theorem \ref{thm:change0-domain-c-ener} now can be carried out similarly to the procedure of Theorem \ref{thm:ge}, hence we omit it. We emphasize that \ref{itm:C1'} is important in ensuring the existence of solutions to \eqref{eq:hj-lb} where the constant $C_\lambda\in\R$ can be arbitrary.\smallskip

\subsection{The proof of the essential criterion of \eqref{eq:assump_ratio}}\label{ss-proof}

As we mentioned in subsection \ref{ss}, for any $c\in\R$, there exists a unique $h_\lambda(c)$ such that $c$ is the erogdic constant of the Hamiltonian $(x,p)\mapsto H\big((1+r(\lambda)x, p, h_\lambda(c))\big)$. As in Lemma \ref{lem:admis}, we conclude that:
\begin{itemize}

    \item[(i)] $h_\lambda(c)$ is the maximal $a\in \R$ such that the following equation can be solved:
    \[
     H\big((1+r(\lambda))x, Du(x), a\big) \leq c,\quad x\in\Om.
     \]
     
    \item[(ii)] $h_\lambda(c(\lambda)) = 0$ where $c(\lambda)$ is the ergodic constant of $H(x,p,0)$ on $\Omega_\lambda$ (see \eqref{eq:mane-cri-lambda}).
    
    \item[(iii)]  $0 < h_\lambda(c_1) - h_\lambda(c_2) \leq \kappa^{-1}(c_1-c_2)$ if $c_2 <c_1$.
    
\end{itemize}
Let $\mathcal{M}_{c}^\lambda$ be the set of Mather measures associated with  $H\big((1+r(\lambda)x, p, h_\lambda(c))\big)$ and assume $C_\lambda \to c(H)$ as $\lambda \to 0^+$, then due to a similar argument as in Lemma \ref{lem:Mc_conv} we 
see that any accumulating measure $\mu$ of a sequence $\mu_\lambda \in \mathcal{M}_{C_\lambda}^\lambda$ as $\lb\rightarrow 0_+$ has to be contained in $\cM$. With these preparations, we now prove:

\begin{proof}[Proof of Theorem \ref{thm:characterization-zeta}] We have $h_\lambda(c(\lambda)) = 0$, thus $h_\lambda(C_\lambda)\to 0$ as $\lambda\to 0^+$. We have
\begin{align}
&\begin{cases}
\begin{aligned}
    &\int_{T\overline{\Omega}} \Big(L\big((1+r(\lambda))x,v, h_\lambda(C_\lambda)\big) + C_\lambda\Big)\;d\mu(x,v) \geq 0 &\qquad 
    \text{for all}\;\mu \in \mathcal{M},\\
    &\int_{T\overline{\Omega}} \Big(L\big(x,v, 0\big) + c(H)\Big)\;d\mu(x,v) = 0 &\qquad 
    \text{for all}\;\mu \in \mathcal{M}.
\end{aligned}
\end{cases}\label{eq:geq}\\
&\begin{cases}
\begin{aligned}
    &\int_{T\overline{\Omega}} \Big(L\big((1+r(\lambda))x,v, h_\lambda(C_\lambda)\big) + C_\lambda\Big)\;d\nu_\lambda(x,v) = 0 &\quad 
  {\rm for\; all}\; \nu_\lambda \in \mathcal{M}^\lambda_{C_\lambda},\\
    &\int_{T\overline{\Omega}} \Big(L\big(x,v, 0\big) + c(H)\Big)\;d\nu_\lambda(x,v) \geq 0 &\quad 
   {\rm for\; all}\; \nu_\lambda \in \mathcal{M}^\lambda_{C_\lambda}.
\end{aligned}
\end{cases}\label{eq:leq}
\end{align}
Using \ref{itm:L}, \ref{itm:C7} in \eqref{eq:geq}, \eqref{eq:leq} we have
\begin{equation*}
    \eta \int_{T\overline{\Omega}} \left\langle\partial_x L(x,v,0), x\right\rangle d\mu + \liminf_{\lambda\to 0^+} 
     \big(
     \int_{T\overline{\Omega}} 
     \frac
     {
        L(x,v,h_\lambda(C_\lambda))    - L(x,v,0) 
     }{
        \lambda
     }d\mu + 
     \frac{C_\lambda - c(H)}{\lambda} 
     \big) \geq 0
\end{equation*}
for all $\mu \in \mathcal{M}$. Similarly, for any $\nu_{\lb}\in \cM_{C_\lb}^\lb$ converging to $\nu$ as  $\lambda\to 0^+$, there holds
\begin{equation*}
     \eta \int_{T\overline{\Omega}} \left\langle\partial_x L(x,v,0), x\right\rangle d\nu + \limsup_{\lambda\to 0^+} 
     \big(
     \int_{T\overline{\Omega}} 
     \frac
     {
        L(x,v,h_\lambda(C_\lambda))    - L(x,v,0) 
     }{
        \lambda
     }d\nu + 
     \frac{C_\lambda - c(H)}{\lambda} 
     \big) \leq 0
\end{equation*}
since $\nu\in\cM$. Consequently, we deduce that
\begin{equation*}
\begin{aligned}
     &\lim_{\lambda\to 0^+} 
     \left(
     \int_{T\overline{\Omega}} 
     \frac
     {
        L(x,v,h_\lambda(C_\lambda))    - L(x,v,0) 
     }{
        \lambda
     }d\nu + 
     \frac{C_\lambda - c(H)}{\lambda} 
     \right) = \eta \int_{T\overline{\Omega}} \left\langle\partial_x L(x,v,0), -x\right\rangle d\nu.
\end{aligned}    
\end{equation*}
Moreover, by using \ref{itm:C6} we deduce that 
\begin{equation}\label{eq:general-ratio}
\begin{aligned}
     &\lim_{\lambda\to 0^+} 
     \left(
     \frac{C_\lambda - c(H)}{\lambda} +
     \frac{h_\lambda(C_\lambda)}{\lambda}\int_{T\overline{\Omega}} 
     \partial_uL (x,v,0)\; d\nu
     \right) = \eta \int_{T\overline{\Omega}} \left\langle\partial_x L(x,v,0), -x\right\rangle d\nu.
\end{aligned}    
\end{equation}
From this and \ref{itm:C1'} we observe that 
\begin{equation}\label{eq:equivalent-zeta-rho}
    \zeta = \lim_{\lambda\to 0^+} \frac{C_\lambda - c(H)}{\lambda}\;\text{exists} \qquad\Longleftrightarrow \qquad \rho = \lim_{\lambda \to 0^+} \frac{h_\lambda(C_\lambda)}{\lambda}\;\text{exists}.
\end{equation}
Under the existence of $\zeta, \rho$, from \eqref{eq:geq} and \ref{itm:C7} we deduce further that 
\begin{equation*}
    \zeta = \max_{\mu\in \mathcal{M}} \left\lbrace -\rho\int_{T\overline{\Omega}} \partial_u L(x,v,0)\;d\mu +  \eta \int_{T\overline{\Omega}} \partial_x \langle L(x,v,0), -x\rangle\;d\mu \right\rbrace 
\end{equation*}
with the maximum achieved at $\mu = \nu$.
\end{proof}

\begin{rmk}\label{rmk:gene-case} In particular,  if we divide
\eqref{eq:geq} and \eqref{eq:leq} by $r(\lb)$ then make $\lb\rightarrow 0^+$ we can get a twofold conclusion:
\begin{equation*}
    \begin{aligned}
    &\lim_{r(\lambda)\to 0^+} 
     \left(
     \frac{C_\lambda - c(H)}{r(\lambda)} +
     \frac{h_\lambda(C_\lambda)}{r(\lambda)}\int_{T\overline{\Omega}} 
     \partial_uL (x,v,0)\; d\nu^+_0
     \right) = \int_{T\overline{\Omega}} \left\langle\partial_x L(x,v,0), -x\right\rangle d\nu^+_0,\\
     &\lim_{r(\lambda)\to 0^-} 
     \left(
     \frac{C_\lambda - c(H)}{r(\lambda)} +
     \frac{h_\lambda(C_\lambda)}{r(\lambda)}\int_{T\overline{\Omega}} 
     \partial_uL (x,v,0)\; d\nu^-_0
     \right) = \int_{T\overline{\Omega}} \left\langle\partial_x L(x,v,0), -x\right\rangle d\nu^-_0,
    \end{aligned}
\end{equation*}
for certain couple $\nu_0^+, \nu_0^-\in \mathcal{M}$. If $C_\lambda = c(\lambda)$, above two equations lead to the following conclusion in \cite{Tu}:
\[
    \max_{\mu\in \mathcal{M}} \int_{T\overline{\Omega}}  \langle \partial_x L(x,v,0),x\rangle\;d\mu 
  = \lim_
    {\substack{ r(\lb)>0,\\
        \lambda\to 0^+
        }
    } 
    \frac{c(\lambda) - c(H)}{r(\lambda)}, \]
    \[
    \min_{\mu\in \mathcal{M}} \int_{T\overline{\Omega}}  \langle \partial_x L(x,v,0),x\rangle\;d\mu 
      = \lim_
    {
       \substack{ r(\lb)<0,\\
        \lambda\to 0^+}
    } 
    \frac{c(\lambda) - c(H)}{r(\lambda)},
\]
which further implies that \eqref{eq:v-sig} holds if and only if $\lim_{\lb\rightarrow 0^+}\frac{c(\lb)-c(H)}{r(\lb)}$ exists.
\end{rmk}

\begin{proof}[Proof of Corollary \ref{cor:diff-h0(c)-2}] The one-sided and almost everywhere differentiability are clear as $c\mapsto h_0(c)$ is a continuous monotonic function. From the definition of $\mathcal{M},\mathcal{M}_c$ we have
\begin{equation*}
    \begin{aligned}
        &\int_{T\overline{\Omega}} \Big(L(x,v,h_0(c)) + c\Big)d\mu(x,v) \geq 0,\quad 
        \int_{T\overline{\Omega}} \Big(L(x,v,0) + c(H)\Big)d\mu(x,v) = 0,   & &\mu \in \mathcal{M},\\
        &\int_{T\overline{\Omega}} \Big(L(x,v,h_0(c)) + c\Big)d\nu_c(x,v) = 0,\quad 
        \int_{T\overline{\Omega}} \Big(L(x,v,0) + c(H)\Big)d\nu_c(x,v) \geq  0, & &\nu_c\in \mathcal{M}_c.
    \end{aligned}
\end{equation*}
Assume $\nu_c \rightharpoonup \nu_+$ as $c\to c(H)^+$ and $\nu_c\rightharpoonup \nu_-$ as $c\to c(H)^-$, then from \ref{itm:C6} we can respectively evaluate the one-sided derivative of $h_0(c)$ at $c=c(H)$:
\begin{equation*}
\begin{aligned}
    &h'_0\big(c(H)^+\big) = \left(-\int_{T\overline{\Omega}}\partial_uL(x,v,0)d\nu_+(x,v)\right)^{-1} = \min_{\mu\in \mathcal{M}}\left(-\int_{T\overline{\Omega}}\partial_uL(x,v,0)d\mu(x,v)\right)^{-1}\\
    &h'_0\big(c(H)^-\big) = \left(-\int_{T\overline{\Omega}}\partial_uL(x,v,0)d\nu_-(x,v)\right)^{-1} = \max_{\mu\in \mathcal{M}}\left(-\int_{T\overline{\Omega}}\partial_uL(x,v,0)d\mu(x,v)\right)^{-1}.
\end{aligned}
\end{equation*}
The proof is complete.
\end{proof}

Finally, by the same argument as in \eqref{eq:general-ratio} and in the proof of Theorem \ref{thm:change0-domain-c-ener} we obtain Corollary \ref{cor:characterization-zeta-limit-exists}.

\subsection*{Acknowledgment} The work of Jianlu Zhang is supported by the National Key R\&D Program of China (No. 2022YFA1007500) and the National Natural Science Foundation of China (No. 12231010). The authors would like to thank Prof. Hung Vinh Tran for suggesting this topic and for helpful comments. The authors also thank Prof. Hiroyoshi Mitake for patiently checking the earlier version, as well as reinforcing relevant references. Besides, the authors sincerely thank the anonymous referee for his/her comments that greatly improve the presentation.

\appendix
\section{Proofs of some technical results}

\begin{proof}[Proof of Lemma \ref{lem:smooth_smaller_domain}] For (i), if $\Omega$ satisfies \ref{itm:C0''}, we can define $s = \theta^{-1}\delta$. We now show that if $x\in (1-s)\Omega$ then  $\mathrm{dist}(x,\partial\Omega) \geq \delta$. For $y\in \partial\Omega$, we observe that if $\varepsilon>0$ then $(1+\varepsilon)y \in (1+\varepsilon)\partial\Omega$. Let $\varepsilon = s(1-s)^{-1} > 0$, we have
    \begin{equation*}
        (1-s)^{-1}x \in \Omega, \qquad\text{and}\qquad (1-s)^{-1}y \in (1+\varepsilon)\partial\Omega,
    \end{equation*}
and thus by \ref{itm:C0''} we have
\begin{equation*}
    \big|(1-\zeta)^{-1}y  - (1-\zeta)^{-1}x \big| \geq \mathrm{dist}\left((1-\zeta)^{-1}y , \overline{\Omega}\right) \geq \theta \varepsilon.
\end{equation*}
Therefore $|y-x| \geq (1-\zeta)\varepsilon \theta = \zeta \theta = \delta$ by the choice of $\varepsilon$ and $\zeta$ above. Since $y\in \partial\Omega$ is chosen arbitrarily, we obtain $\mathrm{dist}(x,\partial\Omega) \geq \delta$, hence $(1-\theta^{-1}\delta)\subset \Omega^\delta$. \\

For (ii), let $\phi\in \mathrm{C}_c^\infty(\R^n)$ be a smooth function such that $\phi(\cdot) \geq 0$ and $\int_{\R^n} \phi(x)\,dx = 1$. For $\delta>0$ we define $\phi_\delta(x) = \delta^{-n}\phi\left(\delta^{-1}x\right)$ and
\begin{equation*}
    w_\delta(x): = \int_{\Omega}\phi_\delta(x-y)w(y)\;dy = \int_{B(0,\delta)} \phi_\delta(y)w(x-y)\;dy, \qquad x\in \Omega^\delta
\end{equation*}
where $\Omega^\delta = \{x\in \Omega: \mathrm{dist}(x,\partial\Omega) > \delta\}$. As $w$ is Lipschitz, $H(x, Dw(x), 0) \leq c(H)$ for a.e. $x\in \Omega$, therefore using convolution, for every $x\in \Omega^\delta$ we compute
\begin{align}
    H\left(x, Dw_\delta(x),0\right) 
    &= H\left(x, \int_{B(0,\delta)} \phi_\delta(y)Dw(x-y)dy, 0\right) \nonumber\\
    &= H\left(y, \int_{B(0,\delta)} \phi_\delta(y)Dw(x-y)dy, 0\right)  + \varpi_1(C\delta) \nonumber\\
    &\leq \int_{B(0,\delta)}\phi_\delta(y)H(y,Dw(y),0)dy + \varpi_1(C\delta) \leq c(H) + \varpi_1(C\delta) \label{eq:w_delta}
\end{align}
for $x\in \Omega^\delta$ thanks to \ref{itm:C2}, \ref{itm:C4} and the fact that $w$ is Lipschitz.

\begin{itemize}
    \item[$\bullet$] If \ref{itm:C0''} holds then from (i) $\left(1-2\theta^{-1}\delta\right) \Omega \subset \Omega^{2\delta} \subset \Omega^\delta$, thus
    \begin{equation*}
        H(x,Dw_{\delta}(x), 0) \leq c(H) + \varpi_1(C\delta) \qquad\text{in}\; \left(1-2\theta^{-1}\delta\right)\Omega
    \end{equation*}
    where $w_\delta$ satisfies \eqref{eq:w_delta}.
    Let us define
    \begin{equation*}
        \widehat{w}_\delta(x) = \left(\frac{1}{1-2\theta^{-1}\delta}\right)
        w_{\delta}
        \left(
            \left(1-2\theta^{-1}\delta
            \right)x
        \right), \qquad x\in \overline{\Omega}.
    \end{equation*}
    It is clear that $\widehat{w}_\delta\in \mathrm{C}^\infty(\overline{\Omega})$, and 
    \begin{equation*}
        H\left(\left(1-2\theta^{-1}\delta
            \right)x, D\widehat{w}_\delta(x), 0\right) \leq c(H) + \varpi_1(C\delta) \qquad\text{in}\;\Omega.
    \end{equation*}
    Thanks to \ref{itm:C4} and the fact that $w$ is Lipschizt, we deduce that 
    \begin{equation*}
        H\left(x, D\widehat{w}_\delta(x), 0\right) \leq c(H) + 2\varpi_1(C\delta) \qquad\text{in}\;\Omega
    \end{equation*}
    where $C$ is a universal constant that does not depend on $\delta>0$.
    \item[$\bullet$] If $\partial\Omega$ is of class $\mathrm{C}^1$ and \ref{itm:C5} holds then using the smoothness of the unit outward normal vector field $\nu(x):\overline{\Omega}\to \R^n$ we can define a map $T_\delta:\overline{\Omega}\to \overline{\Omega}^{2\delta}$ that maps $\partial\Omega$ to $\partial\Omega^{2\delta}$, $T_\delta(x)=x$ for $x\in \Omega^{4\delta}$ and (see also \cite[Theorem VII.1]{C-DL})
    \begin{equation}\label{eq:est_T_delta}
        \sup_{x\in \overline{\Omega}}
        \big(|T_\delta(x)-x| + |\nabla T_\delta(x) - \mathbf{Id}| \big) \leq C\delta .
    \end{equation}
    Let $\tilde{w}_\delta = w_\delta\circ T_\delta^{-1}$ defined in $\Omega^\delta$,
    where $w_\delta$ satisfies \eqref{eq:w_delta} then $\tilde{w}_\delta(x)\in \mathrm{C}^1(\overline{\Omega})$. Thanks to \ref{itm:C3} and \eqref{eq:est_T_delta} we have $ H(x,D\tilde{w}_\delta(x), 0) \leq c(H) + C\delta + \varpi_1(C\delta)$ in $\Omega$.
\end{itemize}
The proof is complete.
\end{proof}

Next, we give a proof for the comparison principle in Theorem \ref{thm:comp-p}. 

\begin{proof}[Proof of Theorem \ref{thm:comp-p}] Assume that $\max_{\overline{\Omega}}(v_1-v_2) = v_1(x_0) - v_2(x_0) > 0$. For $\varepsilon>0$, let $v_1^\varepsilon(x) = (1+\varepsilon) v_1\left((1+\varepsilon)^{-1}x\right)$ for $x\in (1+\varepsilon)\overline{\Omega}$. Then 
\begin{equation*}
    H\left(\frac{x}{1+\varepsilon}, Dv^\varepsilon_1(x), \frac{v_1^\varepsilon(x)}{1+\varepsilon}\right) \leq 0 \qquad \text{in}\;\Omega.
\end{equation*}
Consider 
\begin{equation*}
    \Phi(x,y) = v^\varepsilon_1\left(x\right)-v_2(y) - \frac{|x-y|^2}{\varepsilon^2}, \qquad (x,y) \in (1+\varepsilon)\overline{\Omega}\times \overline{\Omega}.
\end{equation*}
By compactness of $\overline{\Omega}$, we can assume $\Phi$ achieves its maximum at $(x_\varepsilon,y_\varepsilon) \in (1+\varepsilon)\overline{\Omega}\times \overline{\Omega}$. As $\Phi(x_\varepsilon, y_\varepsilon)\geq \Phi(y_\varepsilon, y_\varepsilon)$ we have
\begin{equation*}
    \frac{|x_\varepsilon-y_\varepsilon|^2}{\varepsilon^2} \leq v_1^\varepsilon\left(x_\varepsilon\right) - v_1^\varepsilon\left(y_\varepsilon\right) \leq (1+\varepsilon)\omega_1(|x_\varepsilon-y_\varepsilon|)
\end{equation*}
where $\omega_1(\cdot)$ is the modulus of continuity of $v_1$. We obtain $|x_\varepsilon-y_\varepsilon| = o(\varepsilon)$ as $\varepsilon\to 0^+$ and thus, from \ref{itm:C0''} we obtain $x_\varepsilon\in (1+\varepsilon)\Omega$ for $\varepsilon$ small. Therefore
\begin{equation*}
    H\left(\frac{x_\varepsilon}{1+\varepsilon}, \frac{2(x_\varepsilon-y_\varepsilon)}{\varepsilon^2}, \frac{v_1^\varepsilon(x_\varepsilon)}{1+\varepsilon}\right) \leq 0 \qquad\text{and}\qquad H(y_\varepsilon, \frac{2(x_\varepsilon-y_\varepsilon)}{\varepsilon^2}, v_2(y_\varepsilon)) \geq 0.
\end{equation*}
By compactness, in the limit we easily see that $(x_\varepsilon, y_\varepsilon)\to (x_0,x_0)$ and 
\begin{equation*}
    \lim_{\varepsilon\to 0} \left[v_1\left(\frac{x_\varepsilon}{1+\varepsilon}\right) - v_2(y_\varepsilon)\right] = v_1(x_0) - v_2(x_0) > 0
\end{equation*}
by our assumption, thus 
\begin{align*}
\kappa \left[v_1\left(\frac{x}{1+\varepsilon} \right) - v_2(y_\varepsilon)\right] \leq H\left(y_\varepsilon, p_\varepsilon, \frac{v_1^\varepsilon(x_\varepsilon)}{1+\varepsilon}\right)  - H(y_\varepsilon, p_\varepsilon, v_2(y_\varepsilon)) 
\end{align*}
where $p_\varepsilon = \frac{2(x_\varepsilon-y_\varepsilon)}{\varepsilon^2}$ is bounded, due to $v_1$ is a subsolution and the coercivity \ref{itm:C3}. From \ref{itm:C4} we have
\begin{align*}
     \left|H\left(\frac{x_\varepsilon}{1+\varepsilon}, p_\varepsilon, \frac{v_1^\varepsilon(x_\varepsilon)}{1+\varepsilon}\right) -  H\left(y_\varepsilon, p_\varepsilon, \frac{v_1^\varepsilon(x_\varepsilon)}{1+\varepsilon}\right) \right| \leq \varpi_1\left(\Vert v_1\Vert_{L^\infty}, \left|\frac{x_\varepsilon}{1+\varepsilon} - y_\varepsilon\right|\left(|p_\varepsilon| + 1\right)\right).
\end{align*}
Combining these facts we obtain
\begin{equation*}
    \kappa \left[v_1\left(\frac{x}{1+\varepsilon} \right) - v_2(y_\varepsilon)\right] \leq \varpi_1\left(\Vert v_1\Vert_{L^\infty}, \left|\frac{x_\varepsilon}{1+\varepsilon} - y_\varepsilon\right|\left(|p_\varepsilon| + 1\right)\right).
\end{equation*}
Let $\varepsilon \to 0$ we obtain a contradiction to the fact that $\max_{\overline{\Omega}} (v_1-v_2) > 0$, therefore the comparison principle holds true $v_1 \leq v_2$ on $\overline{\Omega}$.
\end{proof}

\bibliography{refs.bib}{}
\bibliographystyle{acm}
\end{document}